\documentclass[11pt]{preprint}
\usepackage{amssymb,amsmath,amsthm}
\usepackage{graphicx,psfrag,epsfig,color}
\usepackage{times}
\usepackage{tikz}
\usepackage{microtype}
\usepackage{enumerate}

\usepackage[active]{srcltx}

\newtheorem{theorem}{Theorem}[section]

\newtheorem{lemma}[theorem]{Lemma}

\newtheorem{definition}[theorem]{Definition}

\newtheorem{remark}[theorem]{Remark}

\definecolor{darkred}{rgb}{0.9,0.1,0.1}

\font\bbc=msbm10 scaled 1200

\newcommand{\R}{\mbox {\bbc R}}

\def\eps{{\varepsilon}}

\def\cL{\mathcal{L}}

\def\CC{\mathcal{C}}

\let\d\partial

\def\E{\mathbf{E}}
\def\P{\mathbf{P}}

\numberwithin{equation}{section}

\begin{document}

\title{From averaging to homogenization in cellular flows - an exact description of the transition}
\author{Martin Hairer$^1$, Leonid Koralov$^2$, Zsolt Pajor-Gyulai$^2$}
\institute{University of Warwick, \email{M.Hairer@Warwick.ac.uk}
\and University of Maryland, \email{koralov@math.umd.edu,pgyzs@math.umd.edu}}

\maketitle
\begin{abstract}
We consider a two-parameter averaging-homogenization type elliptic problem together with the stochastic representation of the solution. A limit theorem is derived for the corresponding diffusion process and a precise description of the two-parameter limit behavior for the solution of the PDE is obtained.
\end{abstract}

\section{Introduction}
Let $D_R\subseteq\mathbb{R}^2$ be obtained from a bounded smooth domain $D$ by stretching it by a factor $R$.
Consider the elliptic Dirichlet problem
\begin{equation}\label{elliptic_problem1}
\frac{1}{2}\Delta u^{\varepsilon,R}+\frac{1}{\varepsilon}v\nabla u^{\varepsilon,R}= -f\left(\frac{x}{R}\right)\textrm{ in }D_R,\qquad u^{\varepsilon,R}|_{\partial D_R}=0,
\end{equation}
where  $f$ is a bounded  continuous function on $D$ and $v$ is a smooth incompressible periodic  Hamiltonian vector field. For simplicity, assume that $D$ contains the origin. We further assume that the stream function $H(x_1,x_2)$ such that
\[
v = \nabla^{\perp} H=(-\d_2 H, \d_1 H)\;,
\]
is itself periodic in both variables, that is, the integral of $v$ over the periodicity cell is zero.  We will denote the cell of periodicity by $ \mathcal{T}$, which can be viewed as a unit square or, alternatively, as a torus. Our main additional structural assumption is that the critical points of $H$ are non degenerate and that there is a level set of $H$ (say $H=0$ without loss of generality) that contains some of the saddle points and forms a lattice in $\mathbb{R}^2$, thus dividing the plane into bounded sets that are invariant under the flow (see Figure \ref{fig:Cellular_flow}).
A typical example to keep in mind is the canonical cellular flow given by
$H(x_1,x_2) = \sin(x_1)\sin(x_2)$.

\begin{figure}[center]
\centering
\includegraphics[scale=0.4]{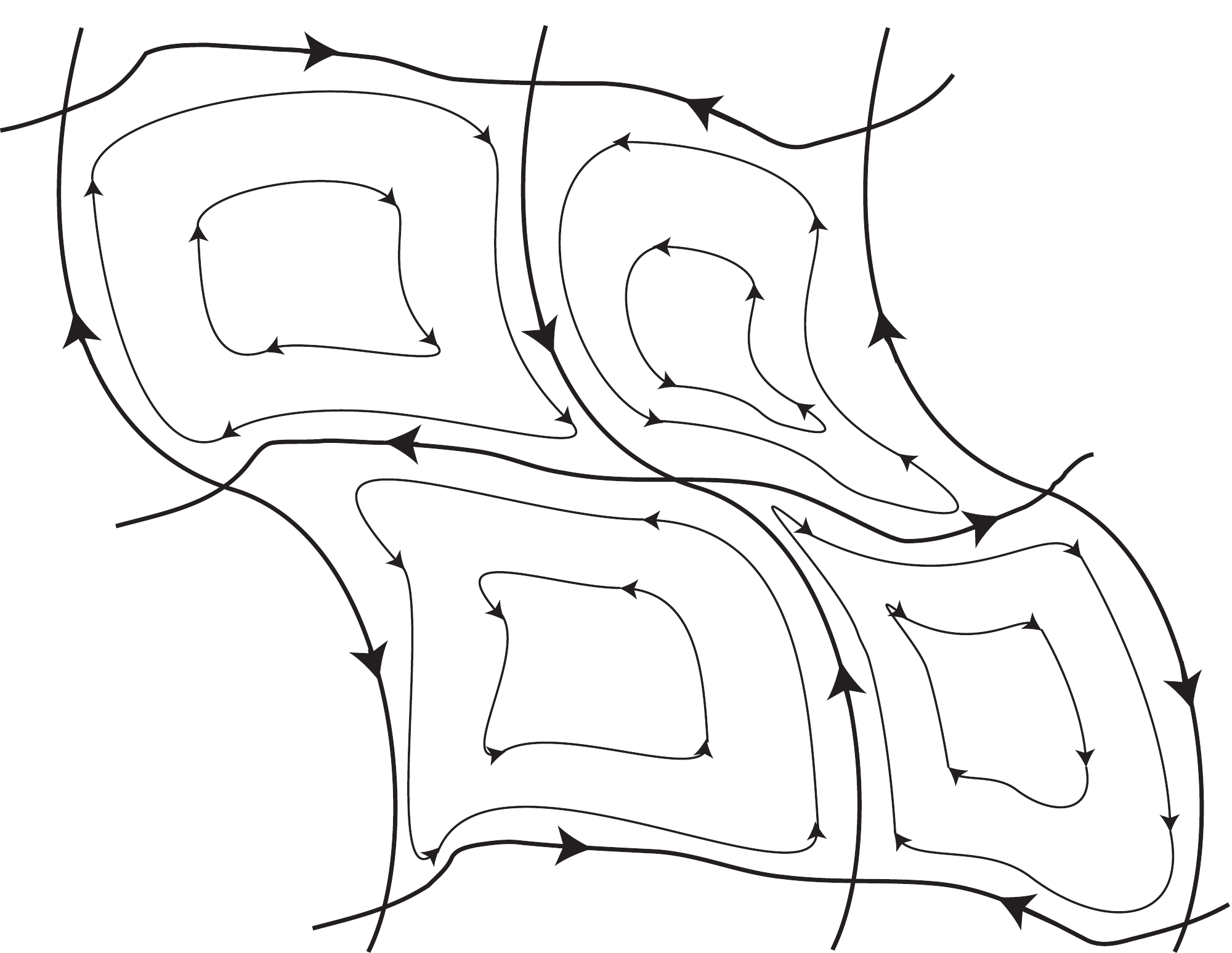}
\caption{A period of the cellular flow}
\label{fig:Cellular_flow}
\end{figure}

There are two parameters in this problem: $\varepsilon$ measures the inverse of the strength of the vector field, while $R$ measures the size of the domain. For fixed $R$ (for example
when $D_R$ coincides with exactly one cell) and $\varepsilon\downarrow 0$,
solution to \eqref{elliptic_problem1} becomes constant on stream lines.
Indeed, multiplying by $\varepsilon$
and letting $\varepsilon\downarrow 0$ formally gives us $v\nabla u=0$. The precise
values of the asymptotics of the solution on each streamline are determined by an
ODE corresponding to the structure of the level sets according to classical averaging results \cite{freidlin2012random}.

If on the other hand $\varepsilon$ is fixed and $R \uparrow \infty$, then the asymptotic behavior of $u$ can be obtained by homogenization (e.g. \cite{papanicolau1978asymptotic,zhikov1994homogenization,pavliotis2008multiscale}), i.e., by solving an elliptic problem on $D$ with appropriately chosen constant coefficients.

It was shown in \cite{Iyer2013} that averaging and homogenization can also be used to study the two-parameter asymptotics in certain regimes. Namely, if $R^4\log^2R\leq c/(\varepsilon\log^2\varepsilon)$ for some constant $c$ as $1/\varepsilon,R\uparrow\infty$, then averaging theory applies. On the other hand, if $R^{4-\alpha}\geq 1/\varepsilon$ for some positive $\alpha$, then homogenization type behavior is observed. The methods in \cite{Iyer2013} are analytic, based on  investigating the asymptotic behavior of the principal Dirichlet eigenvalue of the elliptic operator, and it seems unlikely that they can be directly applied near the transition regime. To our knowledge,  only numerical results were available in the intermediate cases \cite{eps342977,Pavliotis20091030} up until now.

In this paper, we study the two-parameter asymptotics using a probabilistic approach and we prove that the crossover from homogenization to averaging occurs when $R$ is precisely of order $\varepsilon^{-1/4}$. In order to achieve this, we study the family of two dimensional diffusion processes associated to \eqref{elliptic_problem1}, namely
\[
dX_t^{x,\varepsilon}=\frac{1}{\varepsilon}v(X_t^{x,\varepsilon})dt+dW_t,\qquad X_0^{x,\varepsilon}=x,
\]
on some probability space $(\Omega,\mathcal{F},\P)$, where $W_t$ is a two
dimensional Brownian motion.
Our goal is to obtain a limit theorem as $\varepsilon \downarrow 0$ provided that $X_t^{x,\varepsilon}$ is considered on scales of order $\eps^{-1/4}$, and to identify the limiting
process as a time changed Brownian motion. The time change arising in the construction of the
limiting process is non-trivial and can be described as the local time of a diffusion
process on a certain graph which we now explain.

It is well known that there is a graph $G$ naturally associated to the structure of the level sets of $H$ (see Figure \ref{fig:proj2graph}). Namely, let $\mathcal{L}=\{x\in\mathbb{R}^2 ; H(x)=0\}$ be  the connected level set of $H$ that contains a periodic array of saddle points, and denote the corresponding level set on the torus by $\mathcal{L}_{\mathcal{T}}$. Let $A_i$, $i=1,\ldots,n$, be the saddle points of $H$ in $\mathcal{L}_{\mathcal{T}}$. Also, let $U_i$, $i=1,\ldots,n$, be the connected components of $\mathcal{T}\backslash\mathcal{L}_{\mathcal{T}}$. (There is no particular connection between the numbering of the $U_i$'s and that of the $A_i$'s, although by Euler's theorem there is actually the same number of them).  For notational simplicity, assume that there are no loops with only one saddle point in $\mathcal{L}_{\mathcal{T}}$ and that there are no saddle points of $H$ inside any $U_i$. The graph $G$ will then have an interior vertex $O$ and $n$ edges connecting $O$ with the exterior vertices corresponding to the extrema of $H$. This is also called the Reeb graph of $H$.

Define
\[
\Gamma:\mathbb{\mathcal{T}} \to G,\qquad\Gamma(x)=(i,|H(x)|)~~ {\rm if}~x \in \overline{U}_i,
\]
to be the mapping that takes  $U_i$ into an edge $I_i$ of the graph in such a way that the entire set $\mathcal{L}_\mathcal{T}$ is mapped into $O$,  the extrema inside each $U_i$ are mapped into the corresponding exterior vertices, and each connected component of a level set of $H$ is mapped into one point on the corresponding edge of the graph (with each point labeled by the number of the edge and the coordinate on the edge).  Naturally, $\Gamma$ can be extended periodically to the entire plane.

\begin{figure}[h]
\centering
\begin{tikzpicture}
\node[inner sep=0pt] (russell) at (0,0)
    {\includegraphics[scale=0.4]{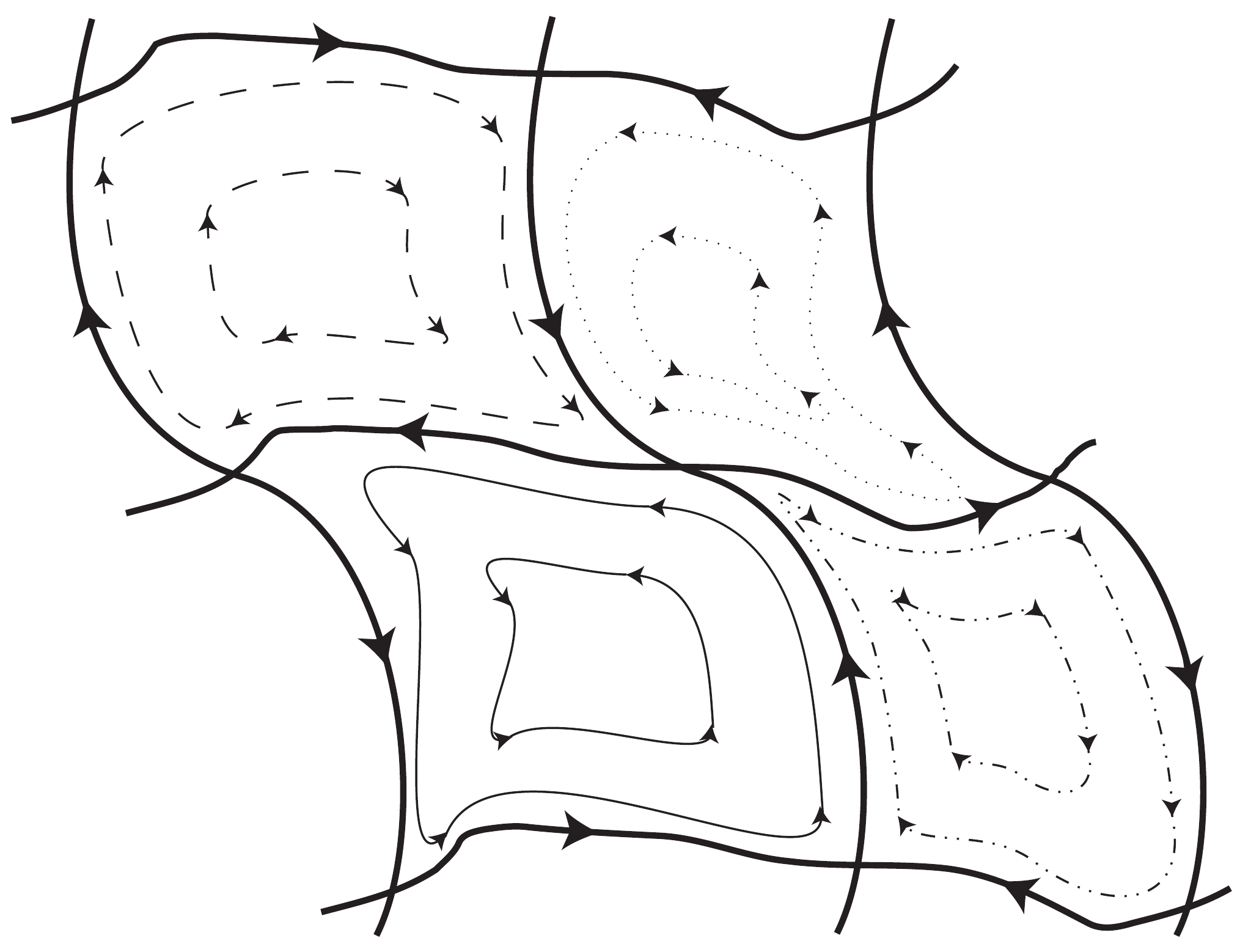}};
\node at (0.6,1) {$\mathbf 2$};
\node at (-2,1.4) {$\mathbf 1$};
\node at (2.6,-1.45) {$\mathbf 4$};
\node at (-0.1,-1.2) {$\mathbf 3$};

\node[circle,inner sep=1mm,fill=black] at (7,0) [label=right:{$O=\Gamma(\mathcal{L})$}] {};
\draw[dashed] (7,0) -- (5.5,3) node[midway,below left] {$\Gamma(\mathbf{1})$};
\draw (7,0) -- (5.5,-3) node[midway,above left] {$\Gamma(\mathbf{3})$};
\draw[dotted] (7,0) -- (8.5,3) node[midway,below right] {$\Gamma(\mathbf{2})$};
\draw[dash pattern=on 3pt off 2pt on .5pt off 2pt on .5pt off 2pt] (7,0) -- (8.5,-3)
 node[midway,above right] {$\Gamma(\mathbf{4})$};
\end{tikzpicture}
\caption{The graph corresponding to the structure of the level sets of $H$ on $\mathcal{T}$}
\label{fig:proj2graph}
\end{figure}
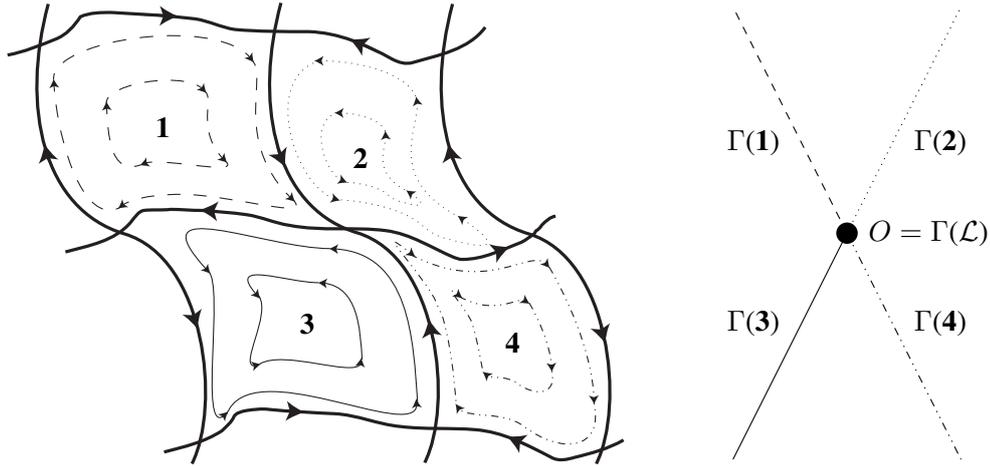

It was shown in \cite[Chapter 8]{freidlin2012random} that the non-Markovian processes $\Gamma(X_t^{x,\varepsilon})$ converge in distribution,  as $\varepsilon \downarrow 0$,  to a diffusion on the graph.  Let us describe this limiting process briefly. On the $i$-th edge of the graph, the process is a
diffusion with generator
\[
L_i=\frac{a(i,y)^2}{2}\frac{d^2}{dy^2}+b(i,y)\frac{d}{dy},
\]
where the coefficients $a(i,y)$, $b(i,y)$ can be  computed explicitly from $H$. The behavior of the process at the interior vertex can also be described in terms of $H$. More precisely, for a set of constants $\alpha_i>0$ with $\sum_{i=1}^n\alpha_i=1$, we can define an operator $A$ on
the domain $D(A)$ that consists of the functions $F$ that satisfy:

\begin{enumerate}
\item[a)] $F \in \CC(G)$ and furthermore $F \in \CC^2(I_i)$ for each edge $i$,
\item[b)] $L_i F(x)$, $x \in I_i$, which is defined on the union of the interiors of all the edges, can be extended to a continuous function on $G$,
\item[c)] $\sum_{i=1}^n\alpha_iD_iF(O)=0$, where $D_iF(O)$ is the one-sided interior derivative of $F$ along the edge $I_i$.
\end{enumerate}

%
%
%
We then define the operator $A$
by $AF|_{I_i}=L_iF|_{I_i}$.
As shown in \cite{freidlin1993}, $A$ generates a Fellerian Markov family $Y_t^y$ on $G$.
With these notations at hand, the measures on $\CC([0,\infty);G)$ induced by the
processes $\Gamma(X^{x,\varepsilon}_t)$ converge weakly to the one induced by the
process $Y^{\Gamma(x)}_t$, provided that the constants $\{\alpha_i\}_{i=1}^n$ are
suitably chosen. Without loss of generality, we may assume that the family $Y^y_t$ is
realized on the same probability space as $X^{x,\varepsilon}_t$, and that these two families
of processes are independent.

Note that the classical Freidlin-Wentzell theory requires $H(x)\to\infty$ as $|x|\to\infty$. Nevertheless, adapting the results for the compact setting on $\mathcal{T}$ is trivial.

\begin{definition}\label{def_local_time}
The local time of $Y^{y_0}$ is the unique nonnegative random field
\[
L^{y_0}=\{L^{y_0}_t(y)\,:\, (t,y)\in[0,\infty)\times G\}
\]
such that the following hold:
\begin{enumerate}
\item The mapping $(t,y)\to L^{y_0}_t(y)$ is measurable and $L^{y_0}_t(y)$ is adapted.
\item For each $y \in G$, the mapping $t\to L^{y_0}_t(y)$ is non-decreasing and constant on each open interval where $Y^{y_0}_t\neq y$.
\item For every Borel measurable $f:G\to [0,\infty)$, we have
\[
\int_0^tf(Y^{y_0}_s)a^2(Y^{y_0}_s)ds=2\int_Gf(y)L^{y_0}_t(y)dy\qquad a.s.
\]
\item $L^{y_0}_t(y)$ is a.s. jointly continuous in $t$ and $y$ for $y\neq O$, while
\[
L_t^{y_0}(O)=\sum_{i=1}^n\lim_{y\to O,~ y\in I_i}L^{y_0}_t(y)\;.
\]
\end{enumerate}
\end{definition}
The existence and uniqueness of local time for diffusions on the real line is relatively well studied. These standard results, together with a straightforward modification of the discussion in Section 2 of \cite{freidlin2000sheu}, give the existence and uniqueness for the local time on the graph. Note (see for example \cite{freidlin2012random}) that for processes on $G$ arising from the averaging of a Hamiltonian system, $a^{-2}(\cdot)$ is locally integrable near the interior vertex, which is sufficient for the method of \cite{freidlin2000sheu} to work.

The main result of this paper is the following. For a positive definite symmetric matrix $Q$, let $\tilde{W}^Q_t$ be a two dimensional  Brownian motion with covariance matrix $Q$. Assume that the families of processes $X^{x,\varepsilon}_t$, $Y^y_t$, and $\tilde{W}^Q_t$ are independent.
Also consider the process $\tilde{W}^Q_{L^{y}_t}$, where $L^y_t = L^{y}_t(O)$ is the local time of $Y^{y}_t$  at the interior vertex.
\begin{theorem}\label{main_thm}
There exists a strictly positive definite matrix $Q$ such that the law of the process $\varepsilon^{1/4}X_t^{x,\varepsilon}$  converges, as $\varepsilon \downarrow 0$, to that of $\tilde {W}^Q_{L^{\Gamma(x)}_t}$.
\end{theorem}

\begin{remark}
One might also consider the process $X_t^{x,\varepsilon}$ on slightly shorter timescales.
 At first glance, this may appear uninteresting since, for a generic starting point $x$, this
would simply lead to a fast rotation on the level set $\{y\,:\, H(y) = H(x)\}$. However,
if we consider a starting point on (or sufficiently close to) the separatrix,  one expects to see
a non-trivial limiting process also at these shorter scales. It is natural to conjecture that
this process, after appropriate re-scaling, is given by $\tilde {W}^Q_{L_t}$, where $L_t$ is the local time of a Brownian motion at the vertex of a star-shaped graph,
with the time further rescaled to account for the logarithmic slow-down near
the saddle points.  A similar process already arose as the scaling limit for heavy-tailed
trap models in \cite{benarous2007}.
\end{remark}

It is well known that the solution of \eqref{elliptic_problem1} can be represented as
\[
u^{\varepsilon,R}(x)= \E\int_0^{\tau_{\partial D_R}(X^{x,\varepsilon}_.)}f(X_s^{x,\varepsilon}/R)\,ds,
\]
where $\tau_{\partial D_R}(\omega)$ is the first hitting of the boundary of $D_R$ by the trajectory $\omega\in \CC([0,T];\mathbb{R}^2)$. The essence of the averaging and transition regimes can be captured by the mechanism of the exit of the process $X_t^{\varepsilon}$ from $D_R$ (see \cite{Iyer2013}).

In the averaging regime, the process $X^{x,\varepsilon}_t$ revolves many times roughly along the flow lines within one cell, but once the separatrix is reached, the process exits $D_R$ quickly (as intuitively follows from the typical fluctuation of the limiting Brownian motion after one notices that the local time immediately becomes non-zero after the process reaches the boundary.)

On the other hand, in the homogenization regime, the interiors of many cells are visited before the process exits $D_R$, and there is enough time for the process $L^{\Gamma(x)}_t$ to start growing nearly linearly in $t$, and therefore an overall Brownian behavior to set in. The mean exit time becomes infinite in the limit.

In the intermediate transition regime, the time required to leave $D_R$ remains finite and is of the same order as the local time, although $L^{\Gamma(x)}_t$ is not directly proportional to $t$ in this regime.
We will apply Theorem~\ref{main_thm} in order to obtain the following asymptotic results for the solution of equation \eqref{elliptic_problem1}.
The precise statement of our results from a PDE perspective
can be summarized by the following theorem.

\begin{theorem}\label{Exit_time_results}
Let  $\varepsilon \downarrow 0$ and $R = R(\varepsilon) \uparrow \infty$  in \eqref{elliptic_problem1}.
\begin{enumerate}
\item (Averaging regime) If $R\varepsilon^{1/4}\downarrow  0$, then
\[
u^{\varepsilon, R}(x) \rightarrow f(0) \cdot \E \overline{\tau}_{0}(Y_{\cdot}^{\Gamma(x)})\;,
\]
where  $\overline{\tau}_{0}$ is the first time when a process on $G$ hits the interior vertex.
\item (Transition regime)
If $R\varepsilon^{1/4}\to C \in(0,\infty)$, then
\[
u^{\varepsilon, R}(x) \to \E \int_0^{\tau_{\partial D}} f (\tilde{W}_{L^{\Gamma(x)}_t}^{Q/C^2} )\, dt\;,
\]
with $Q$ as in Theorem~\ref{main_thm}, where $\tau_{\partial D}$ is the first time the process $\tilde{W}_{L^{\Gamma(x)}_t}^{Q/C^2}$ hits the boundary of $D$.
\item (Homogenization regime) There is a constant $c > 0$ such that
if $R\varepsilon^{1/4}\uparrow\infty$, then
\begin{equation}\label{eq:hom_regime}
\varepsilon^{-1/2} R^{-2}  u^{\varepsilon, R}(x) \to  \E \int_0^{\tau_{\partial D}} f(\tilde{W}^{cQ}_t)\, dt\;,
\end{equation}
where $ \tilde{W}^{cQ}_t$ is a Brownian motion with covariance $c Q$ and $\tau_{\partial D}$ is the first time the process $\tilde{W}^{cQ}_t$ hits the boundary of $D$.
\end{enumerate}
\end{theorem}

\begin{remark}
Note that there is no $x$ dependence on the right hand side of (\ref{eq:hom_regime}). If we scale the problem back to the original domain $D$ and then normalize appropriately, the above result gives us that the limit is the solution of a constant coefficient Dirichlet problem on $D$ evaluated at the origin. To get the values of this solution at another point $x$, we must apply the result to the shifted domain $D-x$. This way we can prove that
\[
(\varepsilon^{1/2}R^2)^{-1}u^{\varepsilon,R}(Rx)\rightarrow \E\int_0^{\tau_{\partial D}}f(x+\tilde{W}^{cQ}_t)\,dt\qquad\textrm{ as }~\varepsilon\downarrow 0,R\uparrow\infty,
\]
which contains the classical homogenization result. Here $\tau_{\partial D}$ is the first time when the process $x+\tilde{W}^{cQ}_t$ hits the boundary of $D$.
\end{remark}

\begin{remark}
Although it is not an aim of the present paper, Theorem~\ref{main_thm} can also be used to derive asymptotics for PDEs with periodic right hand side  and for parabolic problems (using the well known probabilistic representations).  These techniques are suitable  for investigating equations  with non-zero boundary data as well.
\end{remark}

This paper is organized as follows. In Section~\ref{Section_Displacement_Separatrix}, we derive a limit theorem to describe the displacement that occurs when the process leaves the interior of a cell and comes close to the separatrix. This, combined with a L\'evy-type downcrossing representation of the local time at the interior vertex, will help us prove Theorem \ref{main_thm} in Section \ref{secttt}. Section \ref{Section_Exit_Time_Proofs} is dedicated to the proof of Theorem~\ref{Exit_time_results}.

\section{Displacement when the process is near the separatrix}\label{Section_Displacement_Separatrix}
In this section we study the behavior of the process when it is close to the separatrix. The process spends most of the time in the interiors of the cells where no cell changes are possible. However, when the process leaves the cell interior, rapid displacement occurs along the separatrix. We will show what happens during one excursion, i.e., between the time when the process hits the separatrix and the time when it goes back to the interior of the domain (the exact meaning of the latter will be explained below).

First, we need some notations. For any two saddle points, introduce $\gamma(A_i,A_j)$ as the set of points in $\mathcal{L}_{\mathcal{T}}$ that get taken to $A_j$ by the flow $\dot x=v(x)$ and to $A_i$ by the flow $\dot x=-v(x)$.  Since we assumed that the separatrices do not form loops,
we always have $\gamma(A_i,A_i)=\emptyset$.
In a neighborhood of each curve $\gamma(A_i,A_j)$, we can consider a smooth coordinate change $(x_1,x_2)\to(H,\theta)$  defined by the conditions $|\nabla\theta|=|\nabla H|$ and $\nabla\theta\perp\nabla H$ on $\gamma(A_i,A_j)$. This way $\theta$ is defined up to multiplication by $-1$ and up to an additive constant.

Let $V^{\delta}=\{x\in\mathbb{R}^2:|H(x)|\leq \delta\}$ and consider the same change of coordinates in $V^{\delta}\cap\overline{U}_k$, in which case $\theta\in[0,\int_{\partial U_k}|\nabla H|dl]$ and the endpoints of the interval are identified. Using these new coordinates, we can define what it means for the process to pass a saddle point. Namely, let
\[
B(A_i,U_k)=\{x\in V^{\delta}\cap\overline{U}_k\,:\, \theta(x)=\theta(A_i)\}\;,\qquad B(A_i)=\bigcup_{k:A_i\in\partial U_k}B(A_i,U_k).
\]

Let $\pi:\mathbb{R}^2\to \mathcal{T}$ be the quotient map from the plane to the torus and, for simplicity, let us denote $\pi(V^{\delta})$ by $V^{\delta}$ again. Introduce the stopping times $\alpha_0^{x,\delta,\varepsilon}=0$, $\beta_0^{x,\delta,\varepsilon}=\inf\{t\geq 0\,:\,X_t^{x,\varepsilon}\in\mathcal{L}\}$ and recursively define
\[
\alpha_n^{x,\delta,\varepsilon}=\inf\Big\{t\geq\beta_{n-1}^{x,\delta,\varepsilon}\,:\,\pi (X_t^{x,\varepsilon})\in\bigcup_{k\neq i}B(A_k)\cup \partial V^{\delta}~~{\rm if}~\pi(X_{\beta_{n-1}^{x,\delta,\varepsilon}}^{x,\varepsilon})\in\gamma(A_i,A_j)\Big\}
\]
and $\beta_n^{x,\delta,\varepsilon}=\inf\{t\geq\alpha_n^{x,\delta,\varepsilon}: X_t^{x,\varepsilon}\in\mathcal{L}\}$. In other words, $\alpha_n^{x,\delta,\varepsilon}$ is the first time after $\beta_{n-1}^{x,\delta,\varepsilon}$ that the process either hits $\partial V^{\delta}$, or goes past a saddle point different from the one behind $X_{\beta_{n-1}^{x,\delta,\varepsilon}}^{x,\varepsilon}$.

We introduce another pair of sequences of stopping times corresponding to successive visits to $\mathcal{L}$ and $\partial V^\delta$. Namely,
let $\mu^{x,\delta,\varepsilon}_0 = 0$,   $\sigma^{x,\delta,\varepsilon}_0 = \beta_0^{x,\delta,\varepsilon}$, and recursively define
\[
\mu^{x,\delta,\varepsilon}_n  = \inf \{t \geq \sigma^{x,\delta,\varepsilon}_{n-1}\,:\, X_t^{x, \varepsilon} \in \partial V^\delta \},~~~\sigma^{x,\delta,\varepsilon}_n  = \inf \{t \geq \mu^{x,\delta,\varepsilon}_{n}\,:\, X_t^{x, \varepsilon} \in \mathcal{L} \}.
\]

Let
\[
S^{x,\delta, \varepsilon}_n  = X_{\sigma^{x, \delta, \varepsilon}_n}^{x, \varepsilon} - X_{\sigma^{x, \delta, \varepsilon}_{n-1}}^{x, \varepsilon},~n \geq 1,~~~~
T^{x,\delta, \varepsilon}_n = \sigma^{x, \delta, \varepsilon}_n - \mu^{x, \delta, \varepsilon}_n,~n \geq 0,
\]
be the displacement between successive visits to $ \mathcal{L}$ and the time spent on the $n$-th downcrossing of $V^\delta$, respectively. We will use the following notion of uniform weak
convergence for probability measures in the sequel.

\begin{definition}
Given two families of random variables $f^{x, \varepsilon}$ and $g^x$ with values in a metric space $M$ and indexed by a parameter $x$, we will say
that $f^{x,\varepsilon}$ converge to $g^x$ in distribution uniformly in $x$ if
\[
\E\varphi(f^{x,\varepsilon}) \rightarrow \E\varphi(g^x)\;,
\]
as $\varepsilon \rightarrow 0$, uniformly in $x$ for each $\varphi \in \CC_b(M)$.
\end{definition}

Let $\eta^{x, \delta, \varepsilon} $ be the random vector with values in $\{1,\ldots,n\}$ defined by
\[
\eta^{x, \delta, \varepsilon} =  i~~{\rm if}~~ X_{ \mu_1^{x,\delta,\varepsilon}}^{x, \varepsilon} \in U_i,~~i =1,\ldots,n,
\]
i.e., $\eta^{x, \delta, \varepsilon} =  i$ if the process ends up in $U_i$ after the first upcrossing of $V^\delta$.
The main result of this section is

\begin{theorem}\label{CLT_for_first_hit}
There are a  $2\times 2$ non-degenerate matrix $Q$, a vector $(p_1,\ldots,p_n)$, and functions $a(\delta)$, $b_1(\delta),\ldots,
b_n(\delta)$ that go to zero as $\delta \to 0$, such that
\begin{equation} \label{mne}
(\varepsilon^{1/4} S^{x, \delta, \varepsilon}_1, \eta^{x, \delta, \varepsilon}) \rightarrow \big(\sqrt{\delta}(1+a(\delta))\sqrt{\xi} N(0,Q), \eta^\delta\big)
\end{equation}
in distribution as $\eps \downarrow 0$, uniformly in $x \in \mathcal{L}$ for all sufficiently small $\delta > 0$, where $\xi$ is an exponential random variable with parameter one, $N$ is a two dimensional normal with covariance matrix $Q$, independent of $\xi$, and $\eta^\delta$ is a random vector with values in $\{1,\ldots,n\}$ independent of $\xi$ and $N$ such that $\P(\eta^\delta = i) = p_i + b_i(\delta)$.
\end{theorem}

Before proving Theorem~\ref{CLT_for_first_hit}, let us briefly discuss one implication. Let $T^{x, \varepsilon} := T^{x,\delta, \varepsilon}_0$ be the time it takes the process starting at $x$ to reach the separatrix. Let $\bar{T}^y$ be the time it takes the limiting process $Y^y_t$ on the graph
to reach the vertex $O$. By the averaging principle \cite{freidlin1993}, $T^{x, \varepsilon} \rightarrow \bar{T}^{\Gamma(x)}$ in distribution uniformly in $x \in \mathcal{T}$. This, together with Theorem~\ref{CLT_for_first_hit} and the strong Markov property of the process imply the following lemma.

\begin{lemma} \label{indlemma} For fixed $m$ and $\delta$, the random vectors
\[
( T^{x,\delta, \varepsilon}_0,  \varepsilon^{1/4}  S^{x, \delta, \varepsilon}_1, T^{x,\delta, \varepsilon}_1,\ldots,T^{x,\delta, \varepsilon}_{m-1}, \varepsilon^{1/4}  S^{x,\delta, \varepsilon}_m)
\]
converge, as $\varepsilon \downarrow 0$, to a random vector with independent components. The limiting distribution for each of the components
$\varepsilon^{1/4} S^{x, \delta, \varepsilon}_1,\ldots,\varepsilon^{1/4} S^{x, \delta, \varepsilon}_m$ is given by Theorem~\ref{CLT_for_first_hit}, i.e., it is equal to the distribution of
$\sqrt{\delta}(1+a(\delta))\sqrt{\xi} N(0,Q)$. The limiting distribution of $ T^{x,\delta, \varepsilon}_0$ is the distribution of $\bar{T}^{\Gamma(x)}$. The limiting distribution for each of the components  $ T^{x,\delta, \varepsilon}_1,\ldots, T^{x,\delta, \varepsilon}_{m-1}$ is equal to the distribution of $\bar{T}^{\zeta}$, where $\zeta$ is a random initial point for the process on the graph, chosen to be at distance $\delta$ from the vertex $O$, in such a way that $\zeta$ belongs to the $i$-th edge with probability $p_i + b_i(\delta)$.
\end{lemma}

We will prove Theorem~\ref{CLT_for_first_hit} by proving a more abstract lemma on Markov chains with a small probability of termination at each step, and demonstrating that the conditions of the lemma are satisfied in the situation of Theorem~\ref{CLT_for_first_hit}.

Let $M$ be a locally compact separable metric space which can be written as a disjoint union
\[
M = X \sqcup C_1 \sqcup \ldots \sqcup C_n\;,
\]
where the sets $C_i$ are closed. Let $p_{\varepsilon}(x,dy)$, $0\leq\varepsilon\leq\varepsilon^0$, be a family of transition probabilities on $M$
and let $g \in \CC_b(M,\R^2)$. We assume that the following properties hold:

\begin{enumerate}[(1)]
\item $p_0(x, X) = 1$ for all $x \in M$ and $p_\varepsilon(x, X) = 1$ for all $x \in M \setminus X$.
\item $p_0(x,dy)$ is weakly Feller, that is the map $x \mapsto \int_Mf(y)p_0(x,dy)$ belongs
to $\CC_b(M)$ if $f\in \CC_b(M)$.
\item There exist bounded continuous functions $h_1,\ldots,h_n: X \rightarrow [0,\infty)$ such that
\[
\varepsilon^{-\frac{1}{2}} p_\varepsilon(x,C_i) \rightarrow h_i(x),~~{\rm uniformly}~{\rm in}~x \in X,
\]
while $\sup_{x \in X} |\varepsilon^{-\frac{1}{2}} p_\varepsilon(x,C_i)| \leq k$ for some positive constant $k$. We also have
\[
J(x): = h_1(x) + \ldots + h_n(x) > 0~~{\rm for}~ x \in X.
\]
\item $p_{\varepsilon}(x,dy)$ converges weakly to $p_0(x,dy)$ as $\varepsilon \rightarrow 0$, uniformly in $x\in K$ for $K\subseteq X$ compact.
\item The transition functions satisfy a strong Doeblin condition uniformly in~$\varepsilon$. Namely, there exist a probability measure $\eta$ on $X$, a constant $a > 0$, and an integer $m > 0$ such that
\[
p_{\varepsilon}^m(x,A) \geq a \eta(A)~~~{\rm for}~~x \in M,~A \in \mathcal{B}(X),~\varepsilon \in [0, \varepsilon_0].
\]
It then follows that for every $\varepsilon$, there is a unique invariant measure $\lambda^{\varepsilon}(dy)$ on $M$ for $p_{\varepsilon}(x,dy)$, and the associated Markov
chain is uniformly exponentially mixing, i.e., there are $\Lambda>0,c>0$, such that
\[
|p_{\varepsilon}^k(x,A)-\lambda^{\varepsilon}(A)|\leq ce^{-\Lambda k}~~~{\rm for}~{\rm all}~ x \in M,~A \in \mathcal{B}(M),~\varepsilon \in [0, \varepsilon_0].
\]
\item The function $g$ is such that $\int_M g\, d\lambda^{\varepsilon}=0$ for each $\varepsilon \in [0,\varepsilon_0]$.
\end{enumerate}

\begin{lemma} \label{abstractlemma}
Suppose that Assumptions 1--6 above are satisfied and let $Z^{x,\varepsilon}_k$ be the Markov chain on $M$ starting at $x$, with transition function $p_\varepsilon$. Let $\tau = \tau(x,\varepsilon)$ be the first time when the chain reaches the set $C = C_1 \sqcup \ldots \sqcup C_n$. Let $e(Z^{x,\varepsilon}_k) = i$ if $Z^{x,\varepsilon}_k \in C_i$.
Then
\begin{equation} \label{tchh}
\Big( \varepsilon^{\frac{1}{4}} (g(Z^{x,\varepsilon}_1) + \ldots +g(Z^{x,\varepsilon}_\tau)), e(Z^{x,\varepsilon}_\tau) \Big) \rightarrow (F_1, F_2)
\end{equation}
in distribution, uniformly in $x \in X$, where $F_1$ takes values in $ \mathbb{R}^2$, $F_2$ takes values in $\{1,\ldots,n\}$, and $F_1$ and $F_2$ are independent. The random variable $F_1$ is distributed as $(\xi/\int_X J d \lambda^0)^{1\over 2} N(0, \bar{Q})$, where $\xi$ is exponential with parameter one and $\bar{Q}$ is the matrix such that
\[
( g(Z^{x, 0}_1) + \ldots + g (Z^{x, 0}_k)) /\sqrt{k} \rightarrow N(0,\bar{Q})\quad \text{in distribution as $k \rightarrow \infty$.}
\]
The random variable $F_2$ satisfies $ \P(F_2 = i) = \int_X h_i\,d\lambda^0/\int_X J \,d \lambda^0$, $i=1,\ldots,n$.
\end{lemma}

Before we proceed with the proof of Lemma~\ref{abstractlemma}, let us show that it does
indeed implies Theorem~\ref{CLT_for_first_hit}.

\begin{proof}[Proof of Theorem~\ref{CLT_for_first_hit}]
Let $\mathcal{L}_0=\mathcal{L} \backslash\{A \in \mathbb{R}^2: \pi(A) \in \{A_i, i=1,\ldots,n\}\}$.
Define $\bar{M} = \mathcal{L}_0 \sqcup \partial V^\delta$. Let us define a family of transition
functions $\bar{p}_\varepsilon (x, dy)$ on $\bar{M}$. For $x \in \mathcal{L}_0$, we
define $\bar{p}_\varepsilon (x, dy)$ as the
distribution of $X_{\tau}^{x,\varepsilon}$ with $\tau = \mu_1^{x,\delta,\varepsilon} \wedge \beta_1^{x,\delta,\varepsilon}$. In other words, it is the measure
induced by the process stopped when it either reaches the boundary of $V^\delta$ or reaches the separatrix after passing by a saddle point. For $x \in \partial V^\delta$, let $\bar{p}_\varepsilon (x, dy)$ coincide with the
distribution of $X_{\bar \tau}^{x,\varepsilon}$ with $\bar \tau = \beta_0^{x,\delta,\varepsilon}$, i.e., the measure
induced by the process stopped when it reaches the separatrix. Since almost every trajectory of $X_t^{x,\varepsilon}$ that starts outside of the set of saddle points does not contain saddle points, $\bar{p}_\varepsilon$ is indeed a stochastic transition function. Let $\bar{Z}^{x,\varepsilon}_k$ be the corresponding
Markov chain starting at $x \in \bar{M}$.

While we introduced $\bar{M}$ as a subset of $ \mathbb{R}^2$, it is going to be more convenient to keep track of $\pi(\bar{Z}^{x,\varepsilon}_k)$ and the latest displacement separately.
Let $\varphi: \bar{M} \rightarrow M :=  \pi( \bar{M}) \times \mathbb{Z}^2$ map
$x \in \bar{M}$ into $(\pi(x), ([x_1], [x_2]))$ ($[x_1]$ and $[x_2]$ are the integer parts of the first and second coordinates of $x$).
Define the Markov chain $Z^{x,\varepsilon}_k$ on $M$ via
\[
Z^{\pi(x),\varepsilon}_0 = (\pi(x), 0),~~~
Z^{\pi(x),\varepsilon}_k =
(\varphi_1(\bar{Z}^{x,\varepsilon}_k), \varphi_2( \bar{Z}^{x,\varepsilon}_k) - \varphi_2 ( \bar{Z}^{x,\varepsilon}_{k-1})),~~k \geq 1.
\]
Let $X = \pi (\mathcal{L}_0) \times \mathbb{Z}^2 =   (\mathcal{L}_{ \mathcal{T}} \backslash\{A_1,\ldots,A_n\}) \times \mathbb{Z}^2$ and $C_i =
(\pi(V^\delta) \cap U_i) \times \mathbb{Z}^2$. Thus $M = X \sqcup C_1\sqcup\ldots\sqcup C_n$ as required. The transition functions $p_\varepsilon(x, dy)$ are defined as the transition functions for the Markov chain $Z^{x,\varepsilon}_k$.

For $x = (q,\xi) \in M$, define $g((q,\xi)) = \xi \in \mathbb{Z}^2$, which corresponds to the integer part of the displacement during the last
step if the chain is viewed as a process on $ \mathbb{R}^2$.
From the definition of the stopping times $\beta_k^{x,\delta,\varepsilon}$, it follows that
$\varphi_2( \bar{Z}^{x,\varepsilon}_k) - \varphi_2 ( \bar{Z}^{x,\varepsilon}_{k-1})$ can only take a finite number of values (roughly speaking, the
process $X^{x,\varepsilon}_t$ makes transitions from one periodicity cell to a neighboring one or to itself between the times $\beta_k^{x,\delta,\varepsilon}$ and $\beta_{k+1}^{x,\delta,\varepsilon}$). Therefore, $g(Z^{\pi(x),\varepsilon}_k)$ is bounded almost surely, uniformly in $x$ and $k$. Also, it is continuous in the product topology of $\pi( \bar{M}) \times \mathbb{Z}^2$.

The paper \cite{Koralov2004} contains some detailed results on the behavior of the process $X^{x,\varepsilon}_t$ near the separatrix. The main idea behind
those results is that the process can be considered in $(H,\theta)$ coordinates in the vicinity of $ \mathcal{L}$. In those coordinates, after an appropriate re-scaling, the limiting process (as $\varepsilon \rightarrow 0$) is easily identified.

In particular, it was shown in \cite{Koralov2004} (Lemma 2.1 and Section 3) that there is a limiting stochastic transition function $p_0(x, dy)$, and properties (1), (2), and (4)-(6) hold. Property (3) follows from Lemma 4.1 of \cite{Koralov2004}. In fact, Lemma 4.1 of \cite{Koralov2004} implies uniform convergence in property (3) on every compact subset of $X$, but it is easy to see that the compactness assumption in the lemma is not needed.
The functions $h_i(x) = h_i^\delta(x)$ depend on $\delta$ and can be identified as
\[
h_i^\delta(x) =  \lim_{\varepsilon \rightarrow 0} \varepsilon^{-1/2} \P\Big({\rm the}~{\rm process}~{\rm starting}~{\rm at}~X^{x,\varepsilon}_{\alpha_1^{x,\delta,\varepsilon}}~{\rm reaches}~\partial V^\delta \cap U_i~{\rm before}~{\rm reaching}~\cL\Big).
\]
From the arguments in Section 4 of \cite{Koralov2004} it follows that
\[
\int_X h^\delta_i(x) \,d \lambda^0(x) = \delta^{-1}(\bar{p}_i + \bar{b}_i(\delta)),~~i =1,\ldots,n,
\]
where $\bar{p}_i > 0$ and $\bar{b}_i(\delta) \rightarrow 0$ as $\delta \rightarrow 0$. Now Lemma~\ref{abstractlemma} implies that Theorem~\ref{CLT_for_first_hit} holds with
\[
Q = \bar{Q} /(\bar{p}_1 + \ldots + \bar{p}_n)\;,\qquad
p_i = \bar{p}_i/(\bar{p}_1 + \ldots + \bar{p}_n)\;.
\]

Finally, let us show that $\bar{Q}$ is non-degenerate. Assuming by contradiction that this is not the case, there is a unit vector $e \in \mathbb{R}^2$ such that the
function $\bar{g} = (e, g): X \rightarrow \mathbb{R}$ has the property that
\[
\big( \bar{g}(Z^{x, 0}_1) + \ldots + \bar{g} (Z^{x, 0}_k)\big) /\sqrt{k} \rightarrow 0\;,
\]
in distribution as $k \rightarrow \infty$.
It follows from $\int_X\bar{g}\,d\lambda_0=0$ and the arguments in \cite[Thm~11]{bolthausen08}
that this is only possible if there is a function $G \in L^2(X, \lambda^0)$  such that
\[
\bar{g}(x) = G(x) - G(Z^{x,0}_1)\;,
\]
almost surely for $\lambda^0$-almost all $x$.
Recall that $x \in X$ can be written as $x = (q,\xi)$, where
$q \in \pi (\mathcal{L}_0)$ and $ \xi \in \mathbb{Z}^2$. Since $Z^{x,0}_1$ does
not depend on $\xi$, while $\bar{g}(x) = (e, \xi)$, we can write $G(x) = \tilde{G}(q) + (e, \xi)$ for some function $\tilde{G}$. Thus
\begin{equation}\label{e:tildeG}
\tilde{G}(q) = \tilde{G}\big((Z^{x,0}_1)^1\big) + \big(e,(Z^{x,0}_1)^2\big)\;,
\end{equation}
where $(Z^{x,0}_1)^1 \in \pi (\mathcal{L}_0)$ and $ (Z^{x,0}_1)^2 \in \mathbb{Z}^2$.
Thus for $\lambda^0$-almost all $x$, we have $ \tilde{G}(q) =\tilde{G}\big((Z^{x,0}_1)^1\big)$
almost surely on the event $ (e,(Z^{x,0}_1)^2) = 0$. From an 
explicit expression for $p_0(x,dy)$ found in \cite{Koralov2004}, it easily follows
that the distribution of $Z^{x,0}_1$ is
absolutely continuous with respect to $\lambda^0$ for each $x$. Therefore, by the Markov property,
$ \tilde{G}(q) =\tilde{G}\big((Z^{x,0}_k)^1\big)$ almost surely on the event
$ (Z^{x,0}_1)^2 = \ldots = (Z^{x,0}_k)^2 = 0$, for $\lambda^0$-almost all $x$. For sufficiently large $k$, the
(sub-probability) distribution of $(Z^{x,0}_k)^1$ restricted to this event has
a positive density with respect to the
projection of $\lambda^0$ onto $\pi (\mathcal{L}_0)$. (The latter statement is a consequence of the geometry of the flow. Roughly speaking, given two points on the separatrix that belong to the same cell of periodicity, the process $\bar{Z}^{x,0}_k$ can go with positive probability from the first point to an arbitrary neighborhood of the second point without leaving the cell of periodicity.) Therefore, $\tilde{G}$ is $\lambda_0$-almost everywhere constant. By \eqref{e:tildeG}, this implies that
$ (e, (Z^{x,0}_1)^2) = 0$ for $\lambda_0$-almost all $x$. Again by the Markov property, $ (e, (Z^{x,0}_k)^2) = 0$ for $\lambda_0$-almost all $x$ for each $k$. Observe, however, that the process $\bar{Z}^{x,0}_k$ starting at an arbitrary point $x$ on the separatrix, has a positive probability of going to any other cell of periodicity if $k$ is sufficiently large. This yields a contradiction, and thus $\bar{Q}$ is non-degenerate.
\end{proof}

Now let us turn to the proof of Lemma~\ref{abstractlemma}. Let
\[
\Omega = \{ \omega = (x, x_1,\ldots,x_k; i): k \geq 0,~x,x_1,\ldots,x_k \in X, i \in \{1,\ldots,n\} \}
\]
be the space of sequences that start at $x \in X$ and end when the sequence enters $C = C_1 \sqcup \ldots \sqcup C_n$, at which point only the index of the set that the sequence enters is taken into account.
The Markov chain $Z^{x,\varepsilon}_k$ together with the stopping time $\tau$ determine a probability measure $\mu_\varepsilon$ on $\Omega$, namely,
\[
\mu_\varepsilon(x, A_1,\ldots,A_k; i) = \int_{A_1}\ldots\int_{A_k}p_\varepsilon(x, dx_1) p_\varepsilon(x_1, dx_2)\cdots p_\varepsilon(x_{k-1}, dx_k) p_\varepsilon(x_k, C_i),
\]
where $A_1,\ldots,A_k \in \mathcal{B}(X)$.
We introduce another probability measure on $\Omega$ via
\begin{align*}
\nu_\varepsilon&(x, A_1,\ldots,A_k, i) =\\
&= \int_{A_1}\cdots\int_{A_k} e^{-\sqrt{\varepsilon}( J(x)+\ldots+J(x_{k-1})) }\frac{ p_\varepsilon(x, dx_1)}{p_\varepsilon(x, X)} \cdots\frac{p_\varepsilon(x_{k-1}, dx_k)}{p_\varepsilon(x_{k-1}, X)}
\frac{ (1-e^{-\sqrt{\varepsilon} J(x_k)}) h_i(x_k)}{J(x_k)}.
\end{align*}
More precisely, we consider a Markov chain $\tilde{Z}^{x,\varepsilon}_k$ on the state space $X$ with transition function $\tilde{p}_\varepsilon(x, dy) = {p}_\varepsilon(x, dy) /{p}_\varepsilon(x, X)$. We can adjoin the states $\{1,\ldots,n\}$ to the space $X$ and assume that at each step the process may get killed by entering a terminal state $i$ with probability $(1-e^{-\sqrt{\varepsilon} H(x_k)})\frac{h_i(x_k)}{J(x_k)}$, $i =1,\ldots,n$.  Let $\sigma$ be the number of steps after which the process is killed.  Then $\nu_\varepsilon(x, A_1,\ldots,A_k, i)$ is the probability that the chain starting at $x$ visits the sets $A_1,\ldots,A_k$ and then enters the terminal state $i$.

\begin{lemma}\label{lem:big_set}
For every $\delta > 0$ there is $\varepsilon' > 0$ such that for $\varepsilon \leq \varepsilon'$ there is a
 a set $\Omega_\varepsilon$ with $\nu_\varepsilon(\Omega_\varepsilon) \geq 1 - \delta$ such that $d \mu_\varepsilon /d \nu_\varepsilon \in (1-\delta, 1+ \delta)$   on  $\Omega_\varepsilon$.
\end{lemma}

\begin{proof}
To choose the set $\Omega_{\varepsilon}$, note that
\[
\nu^{\varepsilon}(\sigma=k)=\E\left[e^{-\sqrt{\varepsilon}(J(\tilde{Z}^{x,\varepsilon})+\ldots+J(\tilde{Z}^{x,\varepsilon}_{k-1}))}(1-e^{-\sqrt{\varepsilon}J(\tilde{Z}_k^{x,\varepsilon})})\right].
\]
Using the law of large numbers for the Markov chain $\tilde{Z}^{x,\varepsilon}$, which can be applied uniformly in $\varepsilon$ due to the uniform mixing (a consequence of Assumption~5), and the boundedness of $J$ (a consequence of Assumption~3), we conclude that for every $\eta>0$ there is a $k_0$ independent of $\varepsilon$ such that
\[
\P\bigg(\Big|\frac{1}{k}\sum_{j=0}^{k-1}J(\tilde{Z}^{x,\varepsilon}_j)-J_{\varepsilon}\Big|\geq\eta\bigg)\leq \eta
\]
for $k\geq k_0$, where $J_{\varepsilon}=\int_XJ(u)d\lambda^{\varepsilon}(u)$. Therefore
\[
\nu^{\varepsilon}(\sigma< a/\sqrt{\varepsilon})\leq\nu^{\varepsilon}(\sigma<k_0)+\eta+(1-e^{-\sqrt{\varepsilon}\sup_{u\in X}J(u)})\sum_{k=k_0}^{[ a/\sqrt{\varepsilon}]}e^{-\sqrt{\varepsilon}(kJ_{\varepsilon}-k\eta)}.
\]
Since $J_{\varepsilon}\to J_0 > 0$ and since $\eta$ was arbitrary, we have $\nu^{\varepsilon}(\sigma<a/\sqrt{\varepsilon})<\delta/4$ (for all sufficiently small $\varepsilon$) if $a$ is small enough. Similarly one can show that $\nu^{\varepsilon}(\sigma>b/\sqrt{\varepsilon})<\delta/4$ if we choose $b$ to be sufficiently large. We set $\Omega^1_{\varepsilon}=\{\sqrt{\varepsilon}\sigma\in [a,b]\}$. Note that $ \nu_\varepsilon(\Omega^1_\varepsilon) \geq 1 - \delta/2$. Also note that
\[
\nu^{\varepsilon}(\sigma=k, h_i(x_k) < \eta; i) = \E\left[e^{-\sqrt{\varepsilon}\sum_{j=0}^{k-1}J(\tilde{Z}_j^{x,\varepsilon})}(1-e^{-\sqrt{\varepsilon}J(\tilde{Z}_k^{x,\varepsilon})})\chi_{\{h_i(\tilde{Z}_k^{x,\varepsilon})<\eta\}} \frac{h_i(\tilde{Z}_k^{x,\varepsilon})}{J(\tilde{Z}_k^{x,\varepsilon})}\right].
\]
Using the inequality $x^{-1}(1-e^{-cx})<c$ for $x,c>0$, this is less than or equal to $\eta\sqrt{\varepsilon}$. This means that if $\eta > 0$ is choosen small enough, then
\[
\nu^\varepsilon( \sqrt{\varepsilon}\sigma\in [a,b], h_i(x_\sigma) < \eta; i) < \delta/2n~~{\rm for}~{\rm each}~i=1,\ldots,n.
\]

We set $\Omega_\varepsilon^2 = \bigcup_{i=1}^n \{ \sqrt{\varepsilon}\sigma\in [a,b], h_i(x_\sigma) < \eta; i \}$ and $\Omega_\varepsilon = \Omega^1_\varepsilon \setminus \Omega^2_\varepsilon$. Thus $\nu_\varepsilon(\Omega_\varepsilon) > 1 -\delta$.
Observe that
\[
\frac{d\mu_{\varepsilon}}{d\nu_{\varepsilon}}(x,x_1,\ldots,x_k,i)=\frac{p_{\varepsilon}(x,X)\cdots p_{\varepsilon}(x_{k-1},X)}{e^{-\sqrt{\varepsilon}( J(x)+\ldots+J(x_{k-1})) }}\frac{p_{\varepsilon}(x_k,C_i)}{1-e^{-\sqrt{\varepsilon}J(x_k)}}\frac{J(x_k)}{h_i(x_k)}~~~{\rm on}~~\Omega_\varepsilon.
\]
By the definition of $\Omega_\varepsilon$,
it suffices to consider $k(\varepsilon) \in[a/\sqrt{\varepsilon},b/\sqrt{\varepsilon}]$. By the definition of $h_i$ and $J$, the product of the last two fractions converges to $1$ uniformly as $\varepsilon\downarrow 0$ (here we use the definition of $\Omega^2_\varepsilon$). Also note that
\[
\prod_{j=0}^{k(\varepsilon)-1}p_{\varepsilon}(x_j,X)=\prod_{j=0}^{k(\varepsilon)-1}(1-\sqrt{\varepsilon}J(x_j)+o(\sqrt{\varepsilon})) = e^{- \sqrt{\varepsilon} \sum_{j=0}^{k(\varepsilon)-1}J(x_j)+o(1)}
\]
as $\varepsilon\downarrow 0$ provided that $k(\varepsilon) \in[a/\sqrt{\varepsilon},b/\sqrt{\varepsilon}]$,
which implies the desired result.
\end{proof}

\noindent{\it Proof of Lemma~\ref{abstractlemma}.}
Using Lemma \ref{lem:big_set}, we restate Lemma~\ref{abstractlemma} in terms of the Markov chain $\tilde{Z}^{x,\varepsilon}_k$. Note first that $\tilde{Z}^{x,\varepsilon}$ inherits the
strong Doeblin property, which holds uniformly in~$\varepsilon$, i.e.\
\[
\tilde{p}_{\varepsilon}^m(x,A) \geq a \eta(A)~~~{\rm for}~~x \in X,~A \in \mathcal{B}(X),~\varepsilon \in [0, \varepsilon_0].
\]
This implies the exponential mixing, i.e., there are $\Lambda>0,c>0$, such that
\[
|\tilde{p}_{\varepsilon}^k(x,A)-\tilde{\lambda}^{\varepsilon}(A)|\leq ce^{-\Lambda k}~~~{\rm for}~{\rm all}~ x \in X,~A \in \mathcal{B}(X),~\varepsilon \in [0, \varepsilon_0],
\]
where $\tilde{p}_{\varepsilon}$ is the transition function for the chain and $\tilde{\lambda}^{\varepsilon}$ is the invariant measure.

We can also restrict the function $g$ (originally defined on $M$) to the space $X$. We claim that for each $\alpha > 0$ it satisfies
\begin{equation} \label{hll1}
\left|\int_X g d\tilde{\lambda}^{\varepsilon}\right| \leq C \varepsilon^{1/2-\alpha}
\end{equation}
for  some constant $C$ and each $\varepsilon \in [0,\varepsilon_0]$. Indeed, by the exponential mixing,
\[
\left|\int_Xg(y)\tilde{p}_{\varepsilon}^k(x,dy)-\int_Xg(y)\tilde{\lambda}^{\varepsilon}(dy)\right|+
\left|\int_Mg(y)p_{\varepsilon}^k(x,dy)-\int_Mg(y)\lambda^{\varepsilon}(dy)\right|\leq c_1e^{-\Lambda k}
\]
for $x \in X$, $\varepsilon \in (0, \varepsilon_0]$. It is also easy to see by induction that
\begin{equation} \label{g77}
\left|\int_Xg(y)\tilde{p}_{\varepsilon}^k(x,dy)-\int_Mg(y)p_{\varepsilon}^k(x,dy)\right| \leq c_2\sqrt{\varepsilon} k.
\end{equation}
Now we can take  $k=[\varepsilon^{-\alpha}]$ in these two inequalities, proving (\ref{hll1}) since ${\int_Mg(y)\lambda^{\varepsilon}(dy) = 0}$. The same two inequalities with $g$ replaced by an arbitrary bounded continuous function $f$ imply that
\[
\int_Xf(y)\tilde{\lambda}^{\varepsilon}(dy)-\int_Mf(y)\lambda^{\varepsilon}(dy) \rightarrow 0~~{\rm as}~\varepsilon \downarrow 0.
\]
We also know that $\lambda^\varepsilon(M \setminus X) \rightarrow 0$ and $\lambda^{\varepsilon}\Rightarrow\lambda^0$ as $\varepsilon\downarrow 0$, as immediately follows from the properties of $p_\varepsilon$ (the latter statement can be also found in Lemma 2.1 in \cite{Koralov2004}). Therefore,
\[
\int_Xf(y)\tilde{\lambda}^{\varepsilon}(dy)-\int_Xf(y)\lambda^{0}(dy) \rightarrow 0~~{\rm as}~\varepsilon \downarrow 0,
\]
that is $\tilde{\lambda}_{\varepsilon}\Rightarrow\lambda_0$ as $\varepsilon\downarrow 0$.

 Recall that $ \bar{Q}$ is the matrix such that
\[
( g(Z^{x, 0}_1) + \ldots + g (Z^{x, 0}_k)) /\sqrt{k} \rightarrow N(0,\bar{Q})
\]
in distribution as $k \rightarrow \infty$.
Let $ \bar{Q}(\varepsilon)$ be such that
\[
\Big( g(\tilde{Z}^{x, 0}_1) + \ldots + g (\tilde{Z}^{x, 0}_k) - k\int_X g d \tilde{\lambda}_{\varepsilon}\Big) /\sqrt{k} \rightarrow N(0,\bar{Q}(\varepsilon))
\]
in distribution as $k \rightarrow \infty$.
From (\ref{g77}) with $k =1$ and $g$ replaced by an arbitrary bounded continuous function $f$ on $X$ it follows that
$\tilde{p}_{\varepsilon}(x,dy)\stackrel{\varepsilon\to 0}\Rightarrow p_0(x,dy)$ uniformly in $x\in K$ for $K\subseteq X$ compact, since we assumed that the same convergence holds for ${p}_{\varepsilon}(x,dy)$. This and the strong Doeblin property for $\tilde{p}_{\varepsilon}(x,dy)$ easily imply that $\bar{Q}(\varepsilon) \rightarrow ~\bar{Q}$ as $\varepsilon \downarrow 0$ (this was proved in Lemma 2.1 (c) of \cite{Koralov2004} under an additional assumption that $\int_X g d\tilde{\lambda}^{\varepsilon} = 0$, which is now replaced by (\ref{hll1})).

We still have the functions $h_i$ defined on $X$, and we assume that the chain gets killed by entering the state $i \in \{1,\ldots,n\}$ with probability
$ (1-e^{-\sqrt{\varepsilon} J(x)}){h_i(x)}/{J(x)}$. Let $\sigma$ be the time when the chain gets killed. Let the random variable  $\tilde{e}$ be equal to $i$ if the process gets killed by entering the state $i$. Since the function $g$ is bounded, omitting one last term in the sum on the left hand side of (\ref{tchh}) does not affect the limiting distribution. Now we can recast (\ref{tchh}) as follows:
\[
\left( \varepsilon^{\frac{1}{4}} \big(g(\tilde{Z}^{x,\varepsilon}_1) + \ldots +g(\tilde{Z}^{x,\varepsilon}_{\sigma})\big), \tilde{e}\right) \rightarrow (F_1, F_2)
\]
in distribution.  Fix $t \in \mathbb{R}$. Let $\xi$ be an exponential random variable with parameter one on some probability space $(\Omega',P')$, independent of the process. Then for $i \in \{1,\ldots,n\}$, we have  by the definition of $\nu_{\varepsilon}$ and by arguments similar to those in the proof of Lemma~\ref{lem:big_set} that
\begin{align}\label{eq:almost_over1}
\E&\left(e^{i\left<\varepsilon^{1/4}\sum_{j=1}^{\sigma}g(\tilde{Z}_j^{x,\varepsilon}),t\right>} ;\tilde{e} = i\right) =\delta(a,b,\varepsilon)+\\
\nonumber &+\sum_{k=[ a/\sqrt{\varepsilon}]}^{[ b/\sqrt{\varepsilon}]} \E\bigg( \frac{h_i(\tilde{Z}^{x,\varepsilon}_k)}{J(\tilde{Z}^{x,\varepsilon}_k)}
e^{i\left<\varepsilon^{1/4}\sum_{j=1}^{k}g(\tilde{Z}_j^{x,\varepsilon}),t\right>}\P'\bigg(
\sqrt{\varepsilon}\sum_{j =0}^{k-1} J(\tilde{Z}^{x,\varepsilon}_j)  < \xi \leq \sqrt{\varepsilon}\sum_{j = 0}^{k} J(\tilde{Z}^{x,\varepsilon}_j)\bigg) \bigg),
\end{align}
where $\delta(a,b,\varepsilon)\to 0$ as $a \to 0,b\to \infty$ uniformly in $\varepsilon$.

Note that by the law of large numbers,
\[
\E\sum_{k=[ a/\sqrt{\varepsilon}]}^{[ b/\sqrt{\varepsilon}]}\bigg|\P'\bigg(\sum_{j=0}^{k-1}J(\tilde{Z}_j^{x,\varepsilon})<
\frac{\xi}{\sqrt{\varepsilon}}<\sum_{j=0}^{k}J(\tilde{Z}_j^{x,\varepsilon})\bigg)-\sqrt{\varepsilon} e^{-k \tilde{J}_\varepsilon \sqrt{\varepsilon}} J (\tilde{Z}^{x,\varepsilon}_k)\bigg|\to 0
\]
as $\varepsilon\to 0$ uniformly in $0<a<b$, where  $\tilde{J}_{\varepsilon}=\int_XJ(u)d\tilde{\lambda}^{\varepsilon}(u)$. Therefore the main (i.e., second) term on the right hand side of (\ref{eq:almost_over1}) can be replaced by
\begin{equation}\label{eq:almost_over25}
\sqrt{\varepsilon} \sum_{k=[ a/\sqrt{\varepsilon}]}^{[ b/\sqrt{\varepsilon}]} \E\left( {h_i(\tilde{Z}^{x,\varepsilon}_k)}
e^{i\left<\varepsilon^{1/4}\sum_{j=1}^{k}g(\tilde{Z}_j^{x,\varepsilon}),t\right>} \right)e^{-k \tilde{J}_\varepsilon \sqrt{\varepsilon}}
\end{equation}

Uniform exponential mixing also tells us that there is a constant $C$ such that for every $0<k_0<k$ we have
\begin{equation}\label{eq:almost_over3}
\left|\E\left(h_i(\tilde{Z}^{x,\varepsilon}_k)e^{\left<\varepsilon^{1/4}\sum_{j=1}^{k-k_0}g(\tilde{Z}_j^{x,\varepsilon}),t\right>}\right)-\E\left(h_i(\tilde{Z}^{x,\varepsilon}_k)\right)\E\left(e^{\left<\varepsilon^{1/4}\sum_{j=1}^{k-k_0}g(\tilde{Z}_j^{x,\varepsilon}),t\right>}\right)\right|<ce^{-\Lambda k_0}\;.
\end{equation}
It is easy to see that fixing $k_0>0$, i.e., dropping finitely many terms from the sum in the exponent in \eqref{eq:almost_over25} does not change the limit (it only introduces an overall error term of order $\varepsilon^{1/4}$).

From the ergodic theorem, which holds uniformly in $\varepsilon$ by uniform exponential mixing, and the fact that $\tilde{\lambda}^{\varepsilon}\Rightarrow \lambda^0$, it follows that
\begin{equation}\label{eq:almost_over4}
\sup_{k \in [[a/\sqrt{\varepsilon}],[b/\sqrt{\varepsilon}]]} \left| \E\left(h_i(\tilde{Z}^{x,\varepsilon}_k)\right) - \int_Xh_i(u)d\lambda^{0}(u)\right| \rightarrow 0\;,
\end{equation}
as $\varepsilon\downarrow 0$.
Choosing $\alpha<1/4$, it follows from (\ref{hll1}) that
\[
\sup_{k \in [[a/\sqrt{\varepsilon}],[b/\sqrt{\varepsilon}]]} \left|\E\left(e^{i\left<\varepsilon^{1/4}\sum_{j=1}^{k-k_0}g(\tilde{Z}_j^{x,\varepsilon}),t\right>}\right) -
\E\left(e^{i\left<\varepsilon^{1/4}\sum_{j=1}^{k-k_0}\left(g(\tilde{Z}_j^{x,\varepsilon})-
\int_Xgd\tilde{\lambda}^{\varepsilon}\right),t\right>}\right)\right| \rightarrow 0\;,
\]
as $\varepsilon\downarrow 0$.
On the other hand,  we have the following version of the central limit theorem:
\[
\sup_{k \in [[a/\sqrt{\varepsilon}],[b/\sqrt{\varepsilon}]]} \left|
\E\left(e^{i\left<\varepsilon^{1/4}\sum_{j=1}^{k-k_0}\left(g(\tilde{Z}_j^{x,\varepsilon})-
\int_Xgd\tilde{\lambda}^{\varepsilon}\right),t\right>}\right) -\E e^{i\left<\sqrt{k} \varepsilon^{1/4} \cdot N(0,\bar{Q}),t\right>}\right| \rightarrow 0\;,
\]
as $\varepsilon\downarrow 0$,
which holds thanks to the uniform strong Doeblin property and the fact that $\bar{Q}(\varepsilon) \rightarrow \bar{Q}$ as $\varepsilon \downarrow 0$.

Combining this with (\ref{eq:almost_over1}), (\ref{eq:almost_over25}), (\ref{eq:almost_over3}), and (\ref{eq:almost_over4}), and using the fact that $\tilde{J}_{\varepsilon}\to J_0$, we obtain that
\[
\limsup_{\varepsilon\downarrow 0}\left|\E\left(e^{i\left<\varepsilon^{1/4}\sum_{j=1}^{k}g(\tilde{Z}_j^{x,\varepsilon}),t\right>} ;\tilde{e} = i\right)-\frac{\int_Xh_id\lambda^0}{\int_XJd\lambda_0}\int_0^{\infty}\E e^{i\sqrt{s}\langle N(0,\bar{Q}),t\rangle}J_0e^{-sJ_0}ds\right|\leq ce^{-\Lambda k_0}.
\]
Since $t$ and $k_0$ were arbitrary, this implies the desired result.
\qed
\\

We close this section by stating a technical lemma that gives us control over how far away the process wanders during an upcrossing.
Its proof relies on the same arguments as the proof of Lemma~\ref{abstractlemma} considering the maximum of $\sum_{j=1}^{k}g(\tilde{Z}_j^{x,\varepsilon})$ until $\sigma$ and using the invariance principle for Markov chains.

\begin{lemma} \label{displacement}
For each $\eta > 0$ there is $\delta_0 > 0$ such that
\[
\lim_{\varepsilon \downarrow 0} \sup_{x \in \mathbb{R}^2} \P \Big(\varepsilon^{1/4} \sup_{0 \leq t \leq \sigma^{x, \delta, \varepsilon}_1 } |X_{t}^{x,\varepsilon}-x| > \eta\Big) < \eta
\]
whenever $0 < \delta \leq \delta_0$.
\end{lemma}

\section{Proof of Theorem \ref{main_thm}} \label{secttt}

The first step in the proof of Theorem~\ref{main_thm} is to show tightness of the family of measures induced by $\varepsilon^{1/4}(X_t^{x,\varepsilon}-x)$, $0 < \varepsilon \leq 1$, $x \in \mathbb{R}^2$. We will then show the convergence of one-dimensional distributions. The convergence of finite-dimensional distributions (and therefore the statement of the theorem) will then follow from the Markov property.

Define $D_t^{y,\delta}$ to be the number of downcrossings from $\delta$ to $0$ by the trajectory of the process $Y^y_t$ up until time $t$, where we start counting after the first visit to the vertex. Namely,  set $\theta_0^{\delta}=0$,  $\tau_0^{\delta}=\inf\{t\geq0: Y^y_t=0\}$, and recursively define
\[
\theta_n^{\delta}=\inf\{t\geq\tau_{n-1}^{\delta}: Y^y_t=\delta\},~~~ \tau_n^{\delta}=\inf\{t\geq\theta_n^{\delta}: Y^y_t=0\},~~~n \geq 1.
\]
Finally, let $D_t^{y,\delta}=\sup\{n\geq 0: \tau_n^{\delta} \leq t\}$.

\begin{lemma}\label{lemma_downcrossing}
We have
\[
\lim_{\delta \downarrow 0} \E|\delta D_t^{y,\delta}
- L_t^y|  = 0
\]
for each $t>0$ and $y\in G$.
\end{lemma}

The proof of this result is almost identical to  \cite[Section 2]{freidlin2000sheu}, the only difference being the replacement of the condition $a(i,y)\geq c>0$ by the local integrability of $(a(i,y))^{-2}$ (and hence of $(a(i,y))^{-1}$) at the interior vertex. As already noted earlier, this is indeed the case here since our graph process arises from the averaging of a Hamiltonian, see \cite[Chapter 8]{freidlin2012random}, so that $a^{-2}(i,y)$ only diverges logarithmically as $y \to 0$.

For the proof of tightness, we are going to need the following two simple results.

\begin{lemma}\label{lem:technical1}
Let $Z_i$ be a sequence of independent zero mean variables with a common distribution, such that all the moments are finite. Then there exists a universal constant $C$ such that
\[
\P\Big(l^{-1/2}\max_{1\leq m\leq l}|Z_1+\ldots+Z_m|> K\Big) \le C {\E |Z_i|^{10} \over K^{10}}\;,
\]
for all $K > 0$.
\end{lemma}

\begin{proof}
By taking the $10$th power and using Chebyshev's inequality,
\begin{equation}\label{eq:bound_on_max}
\P\left(\max_{1\leq m\leq l}|Z_1+\ldots+Z_m|\geq K \sqrt l\right)\leq {1\over K^{10} l^5} \E\max_{1\leq m\leq l}|Z_1+\ldots+Z_m|^{10}\;.
\end{equation}
Since the $Z_i$ are independent centered random variables, the partial sums form a martingale so that, by Doob's maximal inequality,
\[
\sup_{l \geq 1} \left( l^{-5}\E\max_{1\leq m\leq l}|Z_1+\ldots+Z_m|^{10} \right) \leq\left(\frac{10}{9}\right)^{10}\sup_{l \geq 1} \E\left|\frac{Z_1+\ldots+Z_l}{\sqrt{l}}\right|^{10} \le C \E |Z_i|^{10}\;,
\]
where the last inequality follows from the independence of $Z_i$ and trivial combinatorial considerations. The claim now follows at once.
\end{proof}

\begin{lemma}\label{lem:bound_local_time_moment}
We have $\limsup_{t\to 0}\E(L_t^0/t^{1/2})^n<\infty$ for every $n \in \mathbb{N}$.
\end{lemma}

\begin{proof}
By Lemma 2.3 in \cite{freidlin2000sheu} with $F(y)=|y-O|$ being the distance of $y\in G$ from the interior vertex, we get that
\[
|Y^0_t|=\int_0^ta(i(s),Y^0_s)dW_s+\int_0^tb(i(s),Y^0_s)\,ds+L_t^0\;.
\]
By the uniqueness of the Skorokhod-reflection, see e.g. \cite[Section 3.6.C]{karatzas1991brownian}, we have the representation
\begin{equation} \label{repre}
L_t^0=\max_{0\leq s\leq t}\left(-\int_0^sa(i(s),Y^0_s)\,dW_s-\int_0^sb(i(s),Y^0_s)\,ds\right)\;.
\end{equation}
This implies that there is a standard Brownian motion $B$ such that
\[
\left(\frac{L_t^0}{t^{1/2}}\right)^n\leq C\left(\max_{0\leq s\leq t}| B_{\frac{1}{t}\int_0^s(a(i(s),Y^0_s))^2ds}| + t^{-1/2}\int_0^t|b(i(s),Y^0_s)|\,ds \right)^n\;,
\]
and thus the proof is finished by noting that $a$ and $b$ are bounded on the graph.
\end{proof}

\begin{lemma}\label{lem:thightness}
The family of measures induced by $\{\varepsilon^{1/4}(X_t^{x,\varepsilon}-x)\}_{0<\varepsilon\leq 1, x \in \mathbb{R}^2}$ is tight.
\end{lemma}

\begin{proof}
By the Markov property, it is sufficient to prove that for each $\eta > 0$ there are $r \in (0,1)$ and $\varepsilon_0 > 0$ such that
\begin{equation} \label{teq1}
\P \Big( \sup_{0 \leq t \leq r}  | \varepsilon^{1/4}(X_t^{x,\varepsilon}-x) | > \eta \Big) \leq r \eta\;,
\end{equation}
for all $\varepsilon \leq \varepsilon_0$ and $x \in \mathbb{R}^2$.

Take $Z = \sqrt{\xi} N(0, Q)$ and let $Z^\delta_1$, $Z^\delta_2$, etc.\ be independent identically distributed. Assume that their distribution coincides with the distribution of $\sqrt{\delta} (1 + a(\delta)) Z$, where $a(\delta)$ is the same as in the right hand side of (\ref{mne}).

Applying Lemma~\ref{lem:technical1} with $K = \eta  k^{-1/2}/4$, we see that for a given $\eta>0$, there are $k_0 \in (0,1)$ and $\delta_1>0$ such that
\begin{equation} \label{tt01}
\P \Big(\max_{1 \leq m \leq k/\delta} |Z^\delta_1+\ldots+Z^\delta_m| > \eta/4\Big) \leq k^4 \eta/4\;,
\end{equation}
whenever $k\in (0,k_0)$ and $\delta \in (0,\delta_1)$. From (\ref{tt01}) and Lemma~\ref{indlemma}, it follows that there is $\varepsilon_1(k,\delta) > 0$ such that
\begin{equation}\label{eq:tt02}
\P \Big(\max_{1 \leq m \leq k/\delta} |S^{x,\delta,\varepsilon}_1+\ldots+S^{x,\delta,\varepsilon}_m| > \eta/3\Big) \leq k^4 \eta/3\;,
\end{equation}
provided that $\varepsilon \leq \varepsilon_1(k,\delta)$. It is not difficult to see that this estimate and those below are uniform in $x$. Combining  \eqref{eq:tt02} and Lemma~\ref{displacement}, it now follows that there is $\varepsilon_2(k,\delta) > 0$ such that
\begin{equation} \label{fr1}
\P \Big(\sup_{0 \leq t \leq \sigma^{x, \delta ,\varepsilon}_{[k/\delta]}} \varepsilon^{1/4}|X_t^{x,\varepsilon}-x| > \eta/2\Big) \leq k^4  \eta/2\;.
\end{equation}
provided that $\varepsilon \leq \varepsilon_2(k,\delta)$.

Note that by Lemma \ref{lemma_downcrossing} for a given $\eta > 0$, we can find $r > 0$ and $\delta_2 = \delta_2(r) > 0$ such that
\begin{equation} \label{tt02}
\sup_{y \in G} \P(D_{r}^{y,\delta} \geq r^{1/4}/\delta) <\sup_{y\in G}\P(L_r^{y}\geq r^{1/4})+\eta r/4\leq r^2\E(L_r^0/r^{1/2})^8+\eta r/4 \leq \eta r/3
\end{equation}
if $\delta \leq \delta_2$, where the second inequality follows from the Chebyshev inequality and the strong Markov property, while the last
inequality follows from Lemma~\ref{lem:bound_local_time_moment}. As a consequence of Lemma~\ref{indlemma}, we see that there is $\varepsilon_3(r, \delta)$ such that
\begin{equation} \label{fr2}
\P\Big( \sigma^{x, \delta ,\varepsilon}_{[r^{1/4}/\delta]} < r \Big) \leq \P(D_r^{y,\delta}\geq r^{1/4}/\delta)+\eta r/6
\end{equation}
if $\varepsilon \leq \varepsilon_3(r, \delta)$.

Clearly,
\[
\P \Big( \sup_{0 \leq t \leq r}  | \varepsilon^{1/4}(X_t^{x,\varepsilon}-x) | > \eta \Big)\leq\P\Big(\sigma_{[r^{1/4}/\delta]}^{x,\delta,\varepsilon}<r\Big)+\P\Big(\sup_{0\leq t\leq\sigma_{[r^{1/4}/\delta]}^{x,\delta,\varepsilon}}\varepsilon^{1/4}|X_t^{x,\varepsilon}-x|>\eta\Big)
\]
so that, choosing  $r > 0$ sufficiently small, combining (\ref{fr1}) with $k = r^{1/4}$, (\ref{tt02}), and (\ref{fr2}) with  $\delta < \min(\delta_1, \delta_2)$ and $\varepsilon<\min(\varepsilon_1(k, \delta),\varepsilon_2(k,\delta), \varepsilon_3(r,\delta))$, we obtain (\ref{teq1}), which implies tightness.
\end{proof}

For the proof of convergence of one-dimensional distributions, we are going to need a lemma that is a straightforward consequence of tightness.

\begin{lemma}\label{lem:close_stopping}
For $\eta>0$ and  $f\in \CC_b(\mathbb{R}^2)$ uniformly continuous, we can find an $r>0$ such that
\begin{equation} \label{rone}
\sup_{\varepsilon\in(0,1]}|\E f(\varepsilon^{1/4} (X_{\tau''}^{x, \varepsilon}-x) ) - \E f(\varepsilon^{1/4} (X_{\tau'}^{x, \varepsilon}-x) ) | < \eta,
\end{equation}
\begin{equation} \label{rtwo}
|\E f(\tilde{W}^Q_{\tau''}) - \E f(\tilde{W}^Q_{\tau'})| < \eta
\end{equation}
for each pair of stopping times $\tau' \leq \tau''$ that satisfy $ \P (\tau'' > \tau' +r) \leq r$.
\end{lemma}

\begin{proof}
By the tightness result above, for each $\alpha > 0$ we can find $r > 0$ such that
\[
\sup_{x \in \mathbb{R}^2} \P \Big( \varepsilon^{1/4} \sup_{0 \leq t \leq r} |X_t^{x,\varepsilon}-x| > \alpha\Big) < \alpha.
\]
Using that $f$ is uniformly continuous, we can choose $\alpha(\eta)$ small enough so that we can write
\[
\E|f(\varepsilon^{1/4} (X_{\tau''}^{x, \varepsilon}-x) )- f(\varepsilon^{1/4} (X_{\tau'}^{x, \varepsilon}-x)|<\frac{\eta}{3}+\P(\varepsilon^{1/4}|X_{\tau''}^{x,\varepsilon}-X_{\tau'}^{x,\varepsilon}|>\alpha)
\]
 After conditioning on $X_{\tau'}^{x,\varepsilon}$ and using the strong Markov property, the second term is seen to be bounded from above by
\[
\sup_{x \in \mathbb{R}^2} \P\Big(\varepsilon^{1/4}\sup_{0\leq t\leq r}|X_t^{x,\varepsilon}-x|>\alpha\Big)+\P(\tau''-\tau'>r)\leq\alpha+r,
\]
which finishes the proof of (\ref{rone}) once $\alpha$ and $r$ are chosen to be  small enough. The proof of (\ref{rtwo}) is similar.
\end{proof}

Let us fix $t > 0$, $f \in \CC_b( \mathbb{R}^2)$ uniformly continuous, and $\eta > 0$. To show the convergence of one-dimensional distributions, it suffices to prove that
\begin{equation} \label{oddist}
| \E f(\varepsilon^{1/4}(X_t^{x,\varepsilon}-x)) - \E f( \tilde {W}^Q_{L^{\Gamma(x)}_t}) | < \eta
\end{equation}
for all sufficiently small $\varepsilon$. As we discussed in the introduction, the main contribution to $X_t^{x,\varepsilon}$ (found in the first term on the left hand side of (\ref{oddist})) comes from the excursions between $\mathcal{L}$ and $\partial V^\delta$, i.e., the upcrossings of $V^\delta$. Also, the local time in the second term on the left hand side of (\ref{oddist}) can be related to the number of excursions (i.e., upcrossings) between the interior vertex and the set $\Gamma(\{ x: |H(x)| = \delta \})$ on the graph $G$ that happen before time $t$. These two observations will lead us to the proof of (\ref{oddist}).

In order to choose an appropriate value for $\delta$, we need the following lemma (a simple generalization of the CLT).

\begin{lemma} \label{refinedCLT1} Suppose that $N_{\delta}$ are $ \mathbb{N}$-valued random variables independent of the family $\{Z^{\delta}_i \}$ that satisfy $ \E N_{\delta} \leq C/\delta$ for some $C > 0$. Let $f \in \CC_b( \mathbb{R}^2)$ and let $ \tilde{W}^Q_t$ be a Brownian motion with covariance $Q$, independent of $\{N_{\delta}\}$. Then
\[
\E f(Z^{\delta}_1 + \ldots + Z^{\delta}_{N_{\delta}}) - \E f(\tilde{W}^Q_{\delta N_{\delta}}) \rightarrow 0~~{as}~\delta \downarrow 0.
\]
\end{lemma}

Let $e^\delta(t)$ be the (random) time that elapses before the time spent by the process $Y^{y}_{\cdot}$, aside from the upcrossings, equals $t$, i.e.,
\[
e^\delta(t) = t + \sum_{n =1}^\infty (\theta^\delta_n \wedge e^\delta(t) - \tau^\delta_{n-1} \wedge e^\delta(t)).
\]
In other words, we stop a `special' clock every time the process hits the vertex $O$, and re-start it once the process reaches the level set $ \{| y |= \delta \}$. Then $e^\delta(t)$ is the actual time that elapses when the special clock reaches time $t$. Let $N_\delta=N^{y,\delta}_t$ be the number of upcrossings of the interval $[0,\delta]$ by the process $Y^{y}_{\cdot}$ prior to time $e^\delta(t)$.

Similarly, let $e^{\delta, \varepsilon}(t)$ be the time that elapses before the time spent by the process $X^{x, \varepsilon}_t$, aside from the
upcrossings, equals $t$.  Let $N^{x, \delta, \varepsilon}_t$ be the number of upcrossings by the process $X^{x, \varepsilon}_t$ prior to time  $e^{\delta, \varepsilon}(t)$.

\begin{lemma}\label{lem:upcrossing_negligible}
We have $e^{\delta}(t)\to t$ and $\delta(N_t^{y,\delta}-D_t^{y,\delta})\to 0$   in $L^1$ as $\delta\downarrow 0$ for each $y\in G$.
\end{lemma}
\begin{proof} The first statement basically means that most of the time is spent on downcrossings rather than upcrossings.
Its proof  is contained in the proof of Lemma 2.2 in \cite{freidlin2000sheu}. The second statement follows from the first one together with
the Markov property of the process and Lemmas~\ref{lemma_downcrossing} and~\ref{lem:bound_local_time_moment}.
\end{proof}

From Lemmas~\ref{lem:upcrossing_negligible}  and~\ref{lemma_downcrossing} it follows that the conditions of Lemma~\ref{refinedCLT1} are satisfied with our choice of $N_\delta$. We can therefore choose $\delta_0 > 0$ such that
\begin{equation}\label{eq:by_CLT}
\sup_{y\in G}\Big|\E f(Z^{\delta}_1 + \ldots + Z^{\delta}_{N^{\Gamma(x),\delta}_t}) - \E f\Big(\tilde{W}^Q_{\delta N^{\Gamma(x),\delta}_t}\Big)\Big| \leq \eta/10
\end{equation}
whenever $\delta \leq \delta_0$.

Choose $r$ is such that (\ref{rone}) and (\ref{rtwo}) in Lemma~\ref{lem:close_stopping} hold with $\eta/10$ instead of $\eta$. Also, use Lemma \ref{lemma_downcrossing} and Lemma \ref{lem:upcrossing_negligible} to choose $\delta < \delta_0$ sufficiently small so that
\begin{equation}\label{eq:downcrossing_lemma_in_action}
\Big| \E f\Big(\tilde{W}^Q_{\delta D^{\Gamma(x),\delta}_t}\Big) - \E f\Big(\tilde{W}^Q_{L_t^{\Gamma(x)}}\Big)\Big| < \eta/10
\end{equation}
and
\[
\P (\delta N^{\Gamma(x),\delta}_t >  \delta D^{\Gamma(x),\delta}_t +r) \leq r,~~~ \P (e^\delta(t) > t +r) \leq r/2\;.
\]
From the weak convergence of the processes, the latter implies that there is $\varepsilon_0 > 0$ such that
\[
\P (e^{\delta,\varepsilon}(t) > t +r) \leq r
\]
for $\varepsilon < \varepsilon_0$.
By Lemma \ref{lem:close_stopping}, these inequalities imply that
\begin{equation}\label{eq:time_almost_downcrossing}
|\E f(\varepsilon^{1/4} (X_{e^{\delta,\varepsilon}(t)}^{x, \varepsilon}-x) ) - \E f(\varepsilon^{1/4} (X_{t}^{x, \varepsilon}-x) )| < \eta/10\;,
\end{equation}
and
\begin{equation}\label{eq:no_downcrossing_equal_upcrossing}
\Big|\E f\Big(\tilde{W}^Q_{\delta N^{\Gamma(x),\delta}_t}\Big) - \E f\Big(\tilde{W}^Q_{\delta D^{\Gamma(x),\delta}_t}\Big)\Big| < \eta/10\;.
\end{equation}
In what follows $\delta$ is fixed at this value.

Choose $N$ large enough so that
\begin{equation}\label{eq:making_the_number_finite}
|\E f(Z^{\delta}_1 + \ldots + Z^{\delta}_{N^{\Gamma(x),\delta}_t})  - \E f(Z^{\delta}_1 + \ldots + Z^{\delta}_{N^{\Gamma(x),\delta}_t \wedge N}) | < \eta/10
\end{equation}
and by possibly increasing $N$, let $\varepsilon_1>0$  be such that
\begin{equation}\label{eq:finitely_many_downcrossings}
|\E f(\varepsilon^{1/4} (X_{e^{\delta,\varepsilon}(t)}^{x, \varepsilon}-x) ) - \E f(\varepsilon^{1/4} (X_{e^{\delta,\varepsilon}(t) \wedge \sigma^{x,\delta, \varepsilon}_N}^{x, \varepsilon}-x) )| < \eta/10
\end{equation}
for all  $\varepsilon\leq\varepsilon_1$. This latter can be done by noting that by Lemma~\ref{indlemma}, for every $\alpha$ one can select an N such that
\begin{equation}\label{eq:sigma_and_e}
\P(\sigma_N^{x,\delta,\varepsilon}\leq e^{\delta,\varepsilon}(t))<\alpha
\end{equation}
for every small enough $\varepsilon$. Indeed,
\[
\P(\sigma_N^{x,\delta,\varepsilon}\leq e^{\delta,\varepsilon}(t)) = \P(T^{x,\delta, \varepsilon}_1 + \ldots+ T^{x, \delta, \varepsilon}_N \leq t ).
\]
For fixed $N$ and $\delta$, the random variable $T^{x,\delta, \varepsilon}_1 + \ldots+ T^{x, \delta, \varepsilon}_N $ converges in distribution to  some random variable ${\tilde{\tau}}^\delta_N$ as $\varepsilon \downarrow 0$. Choose $N$ large enough so that
\[
\P(\tilde{\tau}^\delta_N \leq t)<\alpha/2,
\]
which implies (\ref{eq:sigma_and_e}). Now we have both $N$ and $\delta$ fixed.

By Lemma~\ref{indlemma}, there is $\varepsilon_2(\delta) > 0$ such that
\begin{equation}\label{eq:conv_of_excursions}
|\E f(\varepsilon^{1/4}(S^{x,\delta,\varepsilon}_1 + \ldots + S^{x,\delta,\varepsilon}_{N^{\Gamma(x),\delta,\varepsilon}_t \wedge N})) -\E f(Z^{\delta}_1 + \ldots + Z^{\delta}_{N^{\Gamma(x),\delta}_t \wedge N}) | <\eta/10
\end{equation}
if $\varepsilon \leq \varepsilon_2$.  It it here where we used the fact that the displacements during upcrossings become independent, in the limit of $\varepsilon \downarrow 0$, from the times spent on downcrossings. We also have that there is an $\varepsilon_3>0$ such that
\begin{equation}\label{eq:excursions_make_up_displacement}
|\E f(\varepsilon^{1/4}(S^{x,\delta,\varepsilon}_1 + \ldots + S^{x,\delta,\varepsilon}_{N^{\Gamma(x),\delta,\varepsilon}_t \wedge N})) -\E f(\varepsilon^{1/4} (X_{e^{\delta,\varepsilon}(t) \wedge \sigma^{x,\delta, \varepsilon}_N}^{x, \varepsilon}-x) )| <\eta/10
\end{equation}
for all  $\varepsilon<\varepsilon_3$.

Collecting (\ref{eq:time_almost_downcrossing}),~(\ref{eq:finitely_many_downcrossings}),~(\ref{eq:excursions_make_up_displacement}),~(\ref{eq:conv_of_excursions}),~(\ref{eq:making_the_number_finite}),~(\ref{eq:by_CLT}),~(\ref{eq:no_downcrossing_equal_upcrossing}) and~(\ref{eq:downcrossing_lemma_in_action}), we obtain (\ref{oddist}) for $\varepsilon\leq\min\{\varepsilon_0,\varepsilon_1,\varepsilon_2,\varepsilon_3\}$, which completes the proof of Theorem~\ref{main_thm}. \qed

\begin{remark}
It is not difficult to show (and it indeed follows from the proof) that convergence in Theorem~\ref{main_thm} is uniform in $x \in \mathbb{R}^2$.
\end{remark}

\section{Proofs of the PDE results}\label{Section_Exit_Time_Proofs}
\noindent{\it Proof of Theorem \ref{Exit_time_results}.} {\it Part 1.}
By the representation formula,
\[
u^{\varepsilon,R}(x)=\E\int_0^{\tau_{\partial D_R}(X_\cdot^{x,\varepsilon})}f(X_s^{x,\varepsilon}/R)\,ds\;,
\]
which can be decomposed as
\[
\E\int_0^{\tau_{\mathcal{L}}(X^{x,\varepsilon}_\cdot)}f(X_s^{x,\varepsilon}/R)ds+ \E \int^{\tau_{\partial D_R}}_{\tau_{\mathcal{L}}(X^{x,\varepsilon}_\cdot)}f(X_s^{x,\varepsilon}/R)\,ds\;,
\]
where $\tau_{\mathcal{L}}$ is the first time the process hits the separatrix. The first term can easily be seen to converge by the averaging theorem to $f(0)\E\bar{\tau}_0(Y^{\Gamma(x)}_\cdot)$, and thus it remains to show that the second term converges to zero. It suffices to show that $\E(\tau_{\partial{D_R}}(X_.^{x,\varepsilon})-\tau_{\mathcal{L}}(X_.^{x,\varepsilon}))\to 0$ as $\varepsilon\to 0$.

Let $ \mathcal{T}$ be the periodicity cell that contains the origin. Recall that $ \mathcal{L}_{\mathcal{T}}$ is the projection of $ \mathcal{L}$ on the torus. Equivalently, we can view it as a set on the plane that is the intersection of $ \mathcal{L}$ and $ \mathcal{T}$.
Thus it is sufficient to show that
\begin{equation} \label{uu1}
\sup_{x \in  \mathcal{L}_{\mathcal{T}}}\E\tau_{\partial D_R}(X_.^{x,\varepsilon})\to 0~~~{\rm as}~\varepsilon \downarrow 0,~R = R (\varepsilon)\;.
\end{equation}
We claim that
\begin{equation} \label{uu2}
\sup_{x \in \mathcal{L}_{ \mathcal{T}}} \P( \tau_{\partial D_R}(X^{x,
\varepsilon}_\cdot) > K ) \rightarrow 0~~~{\rm as}~\varepsilon \downarrow 0,~R = R (\varepsilon)
\end{equation}
for each $K > 0$, and that there is $\varepsilon_0 > 0$ such that
\begin{equation} \label{uu3}
\sup_{\varepsilon \in (0,\varepsilon_0]} \sup_{x \in \mathbb{R}^2} \P(\tau_{ \mathcal{L}} (X^{x,\varepsilon}_\cdot) > 1) < 1\;.
\end{equation}
The latter easily follows from the averaging principle (see \cite{freidlin2012random}, Chapter 8), while the former will be justified below.

Note that
\[
\sup_{x \in \mathcal{L}_{ \mathcal{T}}}
\E\tau_{\partial D_R}(X^{x,
\varepsilon}_\cdot) \leq \int_0^{\infty} \sup_{x \in \mathcal{L}_{ \mathcal{T}}}\P(\tau_{\partial D_R}(X^{x,
\varepsilon}_\cdot) > K)\,dK\;.
\]
By (\ref{uu2}), the integrand tends to zero for each $K$. Also note that the integrand decays exponentially in $K$ uniformly in $\varepsilon$, as
follows from (\ref{uu2}), (\ref{uu3}), and the Markov property of the process. This justifies (\ref{uu1}).

We still need to prove (\ref{uu2}). For a given value of $\delta > 0$ and all sufficiently small $\varepsilon$, we have
\[
\tau_{\partial D_R}(X^{x,
\varepsilon}_\cdot) \leq \tau_{B(0,\delta)}(\varepsilon^{1/4} X^{x,
\varepsilon}_\cdot)\;,
\]
where $\tau_{B(0,\delta)}$ is the time to reach the boundary of the ball of radius $\delta$ centered at the origin.
By Theorem~\ref{main_thm},
\[
\P( \tau_{B(0,\delta)}(\varepsilon^{1/4} X^{x,
\varepsilon}_\cdot) > K) \rightarrow \P( \tau_{B(0,\delta)}(\tilde{W}^Q_{L^0_.}) > K)~~~{\rm as}~ \varepsilon \downarrow 0\;,
\]
since the boundary of the event on the right hand side has probability zero. It remains to note that we can make the right hand
side arbitrarily small by choosing a sufficiently small $\delta$. This is possible
since   $\P(L^0_t > 0) = 1$ for each $t > 0$ (as follows from (\ref{repre}) and the elementary properties of the Brownian motion).

{\it Part 2.} Let's first assume that $f \geq 0$. Observe that for each $t > 0$ we have
\[
\E\int_0^{\tau_{\partial D_R}(X_.^{x,\varepsilon}) \wedge t}f(X_s^{x,\varepsilon}/R)ds=\E\int_0^{\tau_{\partial D }(R^{-1}X^{x,\varepsilon}_.)\wedge t}f\left(R^{-1}X^{x,\varepsilon}_s\right)ds =: \E I_f^t(R^{-1} X^{x,\varepsilon}_\cdot).
\]

By Theorem~\ref{main_thm}, the processes $R^{-1}X^{x,\varepsilon}_.$ converge weakly to $C^{-1}W_{L_\cdot^{\Gamma(x)}}^Q$. Since $I_f^t$ is bounded and is continuous almost surely with respect to the measure induced by $C^{-1}W_{L_\cdot^{\Gamma(x)}}^Q$, we have
\begin{equation} \label{abn}
\E\int_0^{\tau_{\partial D_R}(X_.^{x,\varepsilon}) \wedge t}f(X_s^{x,\varepsilon}/R)\,ds \rightarrow \E\int_0^{\tau_{\partial D}(C^{-1}W_{L_\cdot^{\Gamma(x)}}^Q) \wedge t}f(C^{-1}W_{L_s^{\Gamma(x)}}^Q)\,ds~~{\rm as}~\varepsilon \downarrow 0.
\end{equation}
As in the proof of Part 1, we have that $\P ( \tau_{\partial D_R}(X_.^{x,\varepsilon}) > K)$ decays exponentially in $K$ uniformly in $\varepsilon$, which justifies the fact that we can take $t = \infty$ in (\ref{abn}).
The general case follows by taking $f=f_+-f_-$.

{\it Part 3.} The PDE result easily follows from the weak convergence of the corresponding processes. More precisely, let
$\bar{X}^{x,\varepsilon}_t = R^{-1}(\varepsilon) X^{x,\varepsilon}_{\varepsilon^{1/2}R(\varepsilon)^2 t}$.
We need to show that
\begin{equation} \label{hn1}
\bar{X}^{x,\varepsilon}_\cdot \Rightarrow\tilde{W}^{cQ}_.\qquad~\textrm{as}~\varepsilon\downarrow 0.
\end{equation}

It follows from \cite{Koralov2004}  that
\begin{equation} \label{hn2}
\frac{\varepsilon^{1/4} X_{k \cdot}^{x,\varepsilon}}{\sqrt{k}}\Rightarrow \tilde{W}_\cdot^{D(\varepsilon)} \qquad~\textrm{as}~k\to\infty,
\end{equation}
where $D(\varepsilon)={D}_0+o(1)$ and $D_0$ is a constant multiple of $Q$.
(Strictly speaking, the result in \cite{Koralov2004} concerns the finite dimensional distributions, but the
generalization to the functional CLT is standard in this situation.) Moreover,
it is not difficult to show (by following the proof in \cite{Koralov2004} and using arguments similar to those in the the proof of Lemma \ref{abstractlemma}) that the convergence is uniform in  $\varepsilon$. Therefore, (\ref{hn2}) implies (\ref{hn1}) with $ {c}Q=D_0$. \qed

\subsubsection*{Acknowledgements}

{\small
The authors are grateful to D.\ Dolgopyat, G.\ Iyer, and A.\ Novikov for various helpful suggestions.
While working on the paper, L.\ Koralov was partially supported by the Simons Fellowship in Mathematical Sciences as well as
the NSF grant number 1309084. Z.\ Pajor-Gyulai was partially supported by the NSF grant number 1309084. M.\ Hairer was partially supported by the Royal Society and by the Leverhulme Trust.
}

\bibliography{./citations}
\bibliographystyle{./Martin}

\end{document}